\documentclass{amsart}

\usepackage[foot]{amsaddr}
\usepackage{color,hyperref,mathrsfs,amssymb}
\usepackage[alphabetic]{amsrefs}
\theoremstyle{plain}
\newtheorem{theorem}{Theorem}[section]
\newtheorem{lemma}[theorem]{Lemma}

\newtheorem{coro}[theorem]{Corollary}

\theoremstyle{definition}
\newtheorem{definition}[theorem]{Definition}

\theoremstyle{remark}
\newtheorem{remark}[theorem]{Remark}

\numberwithin{equation}{section}

\newcommand{\abs}[1]{\lvert#1\rvert}

\newcommand{\dual}[2]{\langle\,#1,#2\rangle}
\newcommand{\Lr}[1]{\left(#1\right)}
\newcommand{\lr}[1]{\Bigl(#1\Bigr)}
\newcommand{\set}[2]{\left\{\,#1\,\mid\,#2\,\right\}}

\newcommand{\nm}[2]{\|\,#1\,\|_{#2}}

\newcommand{\hkto}{\hookrightarrow}

\newcommand\al{\alpha}
\newcommand\om{\omega}
\newcommand\R{\mathbb{R}}
\newcommand\md{\mathrm d}
\newcommand\dx{\md\,x}

\newcommand{\mc}[1]{\mathcal{#1}}
\newcommand{\mb}[1]{\mathbb{#1}}
\newcommand{\ms}[1]{\mathscr{#1}}
\newcommand{\mr}[1]{\mathrm{#1}}

\newcommand{\wh}[1]{\widehat{#1}}

\begin{document}
\title[Spectral Barron space for DNN]{Spectral Barron space for deep neural network approximation}
\author[Y. L. Liao\and P. B. Ming]{Yulei Liao \and Pingbing Ming}
\address{LSEC, Institute of Computational Mathematics and Scientific/Engineering Computing, AMSS, Chinese Academy of Sciences, Beijing 100190, China; School of Mathematical Sciences, University of Chinese Academy of Sciences, Beijing 100049, China}
\email{liaoyulei@lsec.cc.ac.cn, mpb@lsec.cc.ac.cn}
\thanks{We would like to thank Professor Renjin Jiang in Capital Normal University for the helpful discussion. We are also grateful to the anonymous referees for their constructive comments, which have greatly improve the quality of this work. This work was funded by National Natural Science Foundation of China through Grant No. 12371438.}
\date{\today}

\begin{abstract}
We prove the sharp embedding between the spectral Barron space and the Besov space with embedding constants independent of the input dimension. Given the spectral Barron space as the target function space, we prove a dimension-free convergence result that if the neural network contains $L$ hidden layers with $N$ units per layer, then the upper and lower bounds of the $L^2$-approximation error are $\mathcal{O}(N^{-sL})$ with $0 < sL\le 1/2$, where $s\ge 0$ is the smoothness index of the spectral Barron space.
\end{abstract}
\maketitle
\section{Introduction}\label{sec:intro}
A series of works have been devoted to studying the neural network approximation error and generalization error with the Barron class~\cites{Barron:1992,Barron:1993,Barron:1994,Barron:2018} as the target function space. For $f$ a complex-valued function and $s\ge 0$, the spectral norm $\upsilon_{f,s}$ is defined as
\[
	\upsilon_{f,s}{:}=\int_{\mathbb{R}^d}\abs{\xi}^s\abs{\wh{f}(\xi)}\mathrm{d}\xi,
\]
where $\widehat{f}$ is the Fourier transform of $f$ in the sense of distribution. A function $f$ is said to belong to the Barron class if the spectral norm $\upsilon_{f,s}$ is finite, and the pointwise Fourier inversion holds true. Nonetheless, it is vital to note that this definition lacks rigor, as it fails to specify the conditions under which the pointwise Fourier inversion is valid. Addressing this issue is a nontrivial matter, as has been discussed in~\cite{Pinsky:1997}. Subsequently, the authors in~\cite{Ma:2017,Xu:2020,Siegel:2022,Siegel:2023} assume $f\in L^1(\mb{R}^d)$ and $s\ge 0$, and define 
\begin{equation}\label{eq:hat-bs}
\wh{\ms{B}}^s(\mb{R}^d){:}=\set{f\in L^1(\mb{R}^d)}{\upsilon_{f,0}+\upsilon_{f,s}<\infty}.
\end{equation}
For functions in $\wh{\ms{B}}^s(\mb{R}^d)$, the Fourier transform and the pointwise Fourier inversion are valid. Unfortunately, we shall prove in Lemma~\ref{lema:hat-bs} that $\wh{\ms{B}}^s(\mb{R}^d)$ equipped with the norm $\upsilon_{f,0}+\upsilon_{f,s}$ is not complete, and thus,  it is not a Banach space.

To tackle this issue, an alternative function spaces has been proposed, which roots tracing back to the seminar work of H\"ormander~\cite{Hormander:1963}. It is defined as follows:
\[
\ms{F}L^s_p(\mb{R}^d){:}=\set{f\in\ms{S}'(\mb{R}^d)}{(1+\abs{\xi}^s)\wh{f}(\xi)\in L^p(\mb{R}^d)}
\]
for $1\le p\le\infty$ and $s\ge 0$. This space has been studied extensively and bears various names. It is sometimes called the H\"ormander space, as mentioned in works such as~\cite{Hormander:1963,Messina:2001,DGSM60:2014,Ivec:2021}; alternatively, it may be referred to  the Fourier Lebesgue space, as evidenced in the works such as~\cite{Grochenig:2002,Pilipovic:2010,BenyiOh:2013,Kato:2020}. We focus on $p=1$ and $s\ge 0$, and denote it as the spectral Barron space:
\[
\ms{B}^s(\mb{R}^d){:}=\set{f\in\ms{S}'(\mathbb{R}^d)}{\upsilon_{f,0}+\upsilon_{f,s}<\infty},
\]
which is equipped with the norm
\[
\nm{f}{\ms{B}^s(\mb{R}^d)}{:}=\upsilon_{f,0}+\upsilon_{f,s}=\int_{\mathbb{R}^d}(1+\abs{\xi}^s)\abs{\wh{f}(\xi)}\mr{d}\xi.
\]
We show in Lemma~\ref{lema:fourinv} that the pointwise Fourier inversion is valid for functions in $\ms{B}^s(\mb{R}^d)$ with a nonnegative $s$. Some authors also refer to $\ms{B}^s(\mb{R}^d)$ as the Fourier algebra or Wiener algebra, whose algebraic properties, such as the Wiener-Levy theorem~\cites{Wiener:1932,Levy:1935,Helson:1959}, have been extensively studied in~\cites{ReitherStegeman:2000,Liflyand:2012}.

Another popular space for analyzing shallow neural networks is the Barron space introduced in~\cite{E:2019,EMW:2022}, which can be viewed as shallow neural networks with infinite width. Recent works such as~\cite{Wojtowytsch:2022,E:2022} claimed that the spectral Barron space is considerably smaller than the Barron space. However, as pointed out in~\cite{Caragea:2023}, this claim lacks accuracy because they have not discriminated the smoothness index $s$ in $\ms{B}^s(\mb{R}^d)$. In addition, the variation space, initially introduced in~\cite{Barron:2008}, has been studied in relation to the spectral Barron space $\ms{B}^s(\mb{R}^d)$ and the Barron space in~\cite{SiegelXu:2022,Siegel:2023}. These spaces have been exploited to study the regularity of partial differential equations~\cite{ChenLuLu:2021,Lu:2021,E:2022,Chen:2023}.  Recently a novel space, originating from variational spline theory, as discussed in~\cite{ParhiNowak:2022}, which is closely related to the variation space, has emerged as a target function space for neural network approximation~\cite{ParhiNowak:2023}.

The first objective of the present work is the analytical properties of $\ms{B}^s(\mb{R}^d)$. In Lemma~\ref{lema:Banach}, we show that $\ms{B}^s(\mb{R}^d)$ is complete, while Lemma~\ref{lema:hat-bs} shows that $\ms{\wh{B}}^s(\mb{R}^d)$ is not complete. This distinction highlights a key difference between these two spaces. Furthermore, Lemma~\ref{lema:ex3} provides an example that illustrates functions in $\ms{B}^s(\mb{R}^d)$ may decay arbitrarily slow. This elegantly constructed example, utilizing the generalized Hypergeometric function, unveils intriguing relationships between the Fourier transform and the decay rate of the functions. Furthermore, we study the relations between $\ms{B}^s(\mb{R}^d)$ and some classical function spaces. In Theorem~\ref{thm:Besov}, we establish the connections between $\ms{B}^s(\mb{R}^d)$ and the Besov space. Moreover, in Corollary~\ref{coro:Sobolev}, we establish the connections between $\ms{B}^s(\mb{R}^d)$ and the Sobolev spaces. 

Notably, we prove the embedding relation
\[
	B^{s+d/2}_{2,1}(\mb{R}^d)\hookrightarrow\ms{B}^s(\mb{R}^d)\hookrightarrow B^s_{\infty,1}(\mb{R}^d),
\]
which is an optimal result seemingly absent from the existing literature. Roughly speaking, the optimality is understood in the sense that $B^{s+d/2}_{2,1}(\mb{R}^d)$ is the largest Besov space contained in $\ms{B}^s(\mb{R}^d)$, while $B^s_{\infty,1}(\mb{R}^d)$ is the smallest Besov space that contains $\ms{B}^s(\mb{R}^d)$. We refer to Thoerem~\ref{thm:Besov} for a precise statement. Moreover, the embedding constants are independent of the input dimension $d$, which indicates that the embedding is effective in high dimension. This embedding may serve as a bridge to study how the Barron space, the variation space and the space in~\cite{ParhiNowak:2022} are related to the classical function spaces such as the Besov space. To the best of our knowledge, only a few studies have addressed the embedding relationships between neural network-related spaces and classical function spaces, such as~\cite{Kutyniok:2022,Grohs:2023b}, while the scope of their investigations has not been comprehensive.

The second objective of this work is to explore the neural network approximation on a bounded domain. Building upon Barron's seminal works on approximating functions in $\ms{B}^1(\mb{R}^d)$ with $L^2$-norm, recent studies have extended the approximation to functions in $\ms{B}^{k+1}(\mathbb{R}^d)$ with $H^k$-norm, as demonstrated in~\cites{Siegel:2020,Xu:2020}. Furthermore, improved approximation rates have been achieved for functions in $\ms{B}^s(\mathbb{R}^d)$ with large $s$
in works such as~\cites{Bresler:2020,MaSiegelXu:2022,Siegel:2022}. These advancements contribute to a deeper understanding of the approximation capabilities of neural networks.

The distinction between deep ReLU networks and shallow networks has been highlighted in the separation theorems presented in~\cite{Eldan:2016,Telgarsky:2016,Shamir:2022,Grohs:2023a}. These theorems offer examples that can be well approximated by deep networks but not by shallow networks without the curse of dimensionality. This sheds light on the differences in the expressive power between the shallow and deep neural networks. Moreover, the approximation rates for neural networks targeting mixed derivative Besov/Sobolev spaces, spectral Barron spaces, and H\"older spaces have also been investigated. These studies contribute to a broader understanding of the approximation capabilities of neural networks in various function spaces as in~\cite{Du:2019,BreslerNagaraj:2020,Bolcskei:2021,LuShenYangZhang:2021,Suzuki:2021}. Additional works focusing on deep neural network approximation in novel spaces may be founded in~\cite{Schmidt-Hieber:2020,E:2022,Kutyniok:2022}. 

We focus on the $L^2$-approximation properties for functions in $\ms{B}^s(\mb{R}^d)$ when $s$ is small. 
In Theorem~\ref{thm:deep}, we establish that a neural network with 
$L$ hidden layers and $N$ units in each layer can approximate functions in $\ms{B}^s(\mb{R}^d)$ with a convergence rate of $\mc{O}(N^{-sL})$ when $0<sL\le 1/2$. This bound is sharp, as proved in Theorem~\ref{thm:lower}.
Our results provide optimal convergence rates compared to existing literature. For deep neural networks, a similar result has been presented in~\cite{BreslerNagaraj:2020} with a convergence rate of $\mc{O}(N^{-sL/2})$. For shallow neural network; i.e., $L=1$, convergence rates of $\mc{O}(N^{-1/2})$ have been established in~\cites{MengMing:2022,Siegel:2022} when $s=1/2$. However, it is worth noting that the constants in their estimates depend on the dimension at least polynomially, or even exponentially, and require other bounded norms besides $\upsilon_{f,s}$. Our results provide a significant advancement by achieving optimal convergence rates without the additional dependency on dimension or other bounded norms.

The remaining part of the paper is structured as follows. In Section 2, we demonstrate that the spectral Barron space is a Banach space and examine its relationship with other function spaces. This analysis provides a foundation for understanding the properties of the spectral Barron space. In Section 3, we delve into the error estimation for approximating functions in the spectral Barron space using deep neural networks with finite depth and infinite width. By investigating the convergence properties of these networks, we gain insights into their approximation capabilities and provide error bounds for their performance. Finally, in Section 4, we conclude our work by summarizing the key findings and contributions of this study. We also discuss potential avenues for future research and highlight the significance of our results in the broader context of function approximation using neural networks. Certain technical results that are tedious but not the main focusing have been postponed to the Appendix.
\section{Completeness of $\ms{B}^s$ and its relation to other function spaces}\label{sec:relation}
This part discusses the completeness of the spectral Barron space and embedding relations to other classical function spaces. Firstly, we fix some notations. Let $\ms{S}$ be the Schwartz space and let $\ms{S}'$ be its topological dual space, i.e., the space of tempered distribution. The Gamma function
\[
	\Gamma(s){:}=\int_0^\infty t^{s-1}e^{-t}\mr{d}t,\qquad s>0.
\]
Denoting the surface area of the unit sphere $\mb{S}^{d-1}$ by $\om_{d-1}=2\pi^{d/2}/\Gamma(d/2)$. The volume of the unit ball is $\nu_d=\om_{d-1}/d$. The Beta function
\[
B(\alpha,\beta){:}=\int_0^1t^{\alpha-1}(1-t)^{\beta-1}\mr{d}t=\dfrac{\Gamma(\alpha)\Gamma(\beta)}{\Gamma(\alpha+\beta)},\qquad \alpha,\beta>0.
\]

The series formulation of the first kind of Bessel function is defined as 
\[
J_\nu(x){:}=(x/2)^\nu\sum_{k=0}^\infty(-1)^k\dfrac{(x/2)^{2k}}{\Gamma(\nu+k+1)k!}.
\]
This definition may be found in~\cite{Luke:1962}*{\S~1.4.1, Eq. (1)}.

For $f\in L^1(\mb{R}^d)$, its Fourier transform of $f$ is defined as
\[
	\wh{f}(\xi){:}=\int_{\mb{R}^d}f(x)e^{-2\pi ix\cdot\xi}\mr{d}x,
\]
and the inverse Fourier transform is defined as
\[
	f^\vee(x){:}=\int_{\mb{R}^d}f(\xi)e^{2\pi ix\cdot\xi}\mr{d}\xi.
\]
If $f\in\ms{S}'(\mb{R}^d)$, then the Fourier transform in the sense of distribution means
\[
\dual{\wh{f}}{\varphi}=\dual{f}{\wh{\varphi}}\qquad \text{for any}\quad\varphi\in\ms{S}(\mb{R}^d)\subset L^1(\mb{R}^d).
\]

We shall frequently use the following Hausdorff-Young inequality. Let $1\le p \le 2$ and $f\in L^p(\mb{R}^d)$, then
\begin{equation}\label{eq:hyineq}
	\nm{\wh{f}}{L^{p'}(\mb{R}^d)}\le \nm{f}{L^p(\mb{R}^d)},
\end{equation}
 where $p'$ is the conjugate exponent of $p$; i.e. $1/p+1/p'=1$.

We shall use the following pointwise Fourier inversion theorem. 
\begin{lemma}\label{lema:fourinv}
	Let $g\in L^1(\mb{R}^d)$, then $\wh{g^\vee}=g$ in $\ms{S}'(\mb{R}^d)$. Furthermore, let $f\in\ms{S}'(\mb{R}^d)$ and $\wh{f}\in L^1(\mb{R}^d)$, then $(\wh{f})^\vee=f$, a.e. on $\mb{R}^d$. 
\end{lemma}

\begin{proof}
By definition, there holds 
	\[
		\dual{\wh{g^\vee}}{\varphi}=\dual{g^\vee}{\wh{\varphi}}
		=\dual{g}{\varphi}\qquad\text{for any}\quad\varphi\in\ms{S}(\mb{R}^d).
	\]
Therefore, $\wh{g^\vee}=g$ in $\ms{S}'(\mb{R}^d)$.	Note that $\wh{f}\in L^1(\mb{R}^d)$, 
	\[
		\dual{(\wh{f})^\vee}{\varphi}=\dual{\wh{f}}{\varphi^\vee}=\dual{f}{\varphi}\qquad\text{for any}\quad\varphi\in\ms{S}(\mb{R}^d).
	\]
	By the Hausdorff-Young inequality~\eqref{eq:hyineq},
	\[
		\nm{(\wh{f})^\vee}{L^\infty(\mb{R}^d)}\le\nm{\wh{f}}{L^1(\mb{R}^d)}.
	\]
Therefore, $f$ is a linear bounded operator on $L^1(\mb{R}^d)$; i.e., $f\in [L^1(\mb{R}^d)]^*=L^\infty(\mb{R}^d)$ due to $\ms{S}(\mb{R}^d)$ is dense in $L^1(\mb{R}^d)$ and
	\[
		\abs{\dual{f}{\varphi}}=\abs{\dual{(\wh{f})^\vee}{\varphi}}\le\nm{(\wh{f})^\vee}{L^\infty(\mb{R}^d)}\nm{\varphi}{L^1(\mb{R}^d)}\le\nm{\wh{f}}{L^1(\mb{R}^d)}\nm{\varphi}{L^1(\mb{R}^d)}. 
	\]
Hence, $(\wh{f})^\vee=f$, a.e. on $\mb{R}^d$ because $(\wh{f})^\vee-f\in L^\infty(\mb{R}^d)$~\cite{Brezis:2011}*{Corollary 4.24}.
\end{proof}

A direct consequence of Lemma~\ref{lema:fourinv} is that the pointwise Fourier inversion is valid for functions in $\ms{B}^s(\R^d)$. We shall frequently use this fact later on.
\subsection{Completeness of the spectral Barron space}
\begin{lemma}\label{lema:Banach}
\begin{enumerate}
\item $\ms{B}^s(\mb{R}^d),s\ge 0$ is a Banach space.

\item When $s>0$, $\ms{B}^s(\mb{R}^d)$ is not a Banach space if the norm $\nm{f}{\ms{B}^s(\R^d)}$ is replaced by $\upsilon_{f,s}$.
\end{enumerate}
\end{lemma}

\begin{proof}
We give a brief proof for the first claim for the readers' convenience, which has been stated in~\cite{Hormander:1963}*{Theorem 2.2.1}.

It is sufficient to check the completeness of $\ms{B}^s(\mb{R}^d)$. For any Cauchy sequence $\{f_k\}_{k=1}^\infty\subset\ms{B}^s(\mb{R}^d)$, there exists $g\in L^1(\mb{R}^d)$ such that $\wh{f}_k\to g$ in $L^1(\mb{R}^d)$. Therefore there exists a sub-sequence of $\{f_k\}_{k=1}^\infty$(still denoted by $f_k$) such that $\wh{f}_k\to g$ a.e. on $\mb{R}^d$. 

	Define the measure $\mu$ by setting that for any measurable set $E\subset\mb{R}^d$,
	\[
		\mu(E){:}=\int_E\abs{\xi}^s\mr{d}\xi.
	\]
	Then $\{\wh{f}_k\}_{k=1}^\infty$ is a Cauchy sequence in $L^1(\mb{R}^d,\mu)$ and there exists $h\in L^1(\mb{R}^d,\mu)$ such that $\wh{f}_k\to h$ in $L^1(\mb{R}^d,\mu)$. Therefore there exists a sub-sequence of $\{f_k\}_{k=1}^\infty$(still denoted by $f_k$) such that $\wh{f}_k\to h$ $\mu$-a.e. on $\mb{R}^d$. Note that for any measurable set $E\subset\mb{R}^d$, $\mu(E)=0$ is equivalent to $\abs{E}=0$. Therefore $\wh{f}_k\to h$ a.e. on $\mb{R}^d$. By the uniqueness of limitation, $h=g$, a.e. on $\mb{R}^d$. 

Define $f=g^\vee$. Lemma~\ref{lema:fourinv} shows that $\wh{f}=g$ in $\ms{S}'(\mb{R}^d)$. Therefore $f\in\ms{B}^s(\mb{R}^d)$ and $f_k\to f$ in $\ms{B}^s(\mb{R}^d)$. Hence $\ms{B}^s$ is complete and it is a Banach space.

We prove (2) by contradiction, suppose that the assertion (2) is false, i.e., $\ms{B}^s(\mb R^d)$ equipped merely with the spectral norm $\upsilon_{f,s}$ is a Banach space, then there exists $C$ depending only on $s$ and $d$ such that for any $f\in\ms{B}^s(\mb{R}^d)$,
\begin{equation}\label{eq:banach1}
\upsilon_{f,0}\le C\upsilon_{f,s}.
\end{equation}
The following example suggests that this is false.

For some $\delta>-1$, let
\[
f_n(x)=\Lr{\sum_{k=1}^n2^{kd}(1-2^{2k}\abs{\xi}^2)_+^\delta}^\vee(x).
\]
We rewrite $f_n$ in terms of the Bochner-Riesz multipliers that is defined by
\[
		\phi_R=\Lr{\Lr{1-\dfrac{\abs{\xi}^2}{R^2}}_+^\delta}^\vee,\qquad\delta>-1.
\]
Therefore, $f_n=\sum_{k=1}^n\phi_{2^{-k}}(x)$. We claim that $\phi_R$ admits the following explicit representation
\begin{equation}\label{eq:phiR}
		\phi_R(x)=\dfrac{\Gamma(\delta+1)}{\pi^\delta\abs{x}^{\delta+d/2}}R^{-\delta+d/2}J_{\delta+d/2}(2\pi\abs{x}R),
\end{equation}
and for $s\ge 0$, the spectrum norm of $\phi_R$ is
\begin{equation}\label{eq:phiRs}
		\upsilon_{\phi_R,s}=\dfrac{\om_{d-1}}2B\Lr{\dfrac{s+d}2,\delta+1}R^{s+d}.
\end{equation}
The proof of the above two identities is postponed to Appendix~\ref{apd:phiR}. It follows from~\eqref{eq:phiR} that
\begin{equation}\label{eq:ex1}
f_n(x)=\dfrac{\Gamma(\delta+1)}{\pi^\delta\abs{x}^{\delta+d/2}}\sum_{k=1}^n2^{k(\delta+d/2)}J_{\delta+d/2}(2^{1-k}\pi\abs{x}),
\end{equation} 
and $f_n\in\ms{B}^s(\mb{R}^d)$ with
	\[
		\upsilon_{f_n,s}=\sum_{k=1}^n2^{kd}\upsilon_{\phi_{2^{-k}},s}
		=\dfrac{1-2^{-ns}}{2^{s+1}-2}\om_{d-1}B\Lr{\dfrac{s+d}2,\delta+1},
	\]
	and
	\[
		\upsilon_{f_n,0}=\sum_{k=1}^n2^{kd}\upsilon_{\phi_{2^{-k}},0}
		=\dfrac{\om_{d-1}}2B\Lr{\dfrac{d}2,\delta+1}n.
	\]
 where we have used~\eqref{eq:phiRs}.
It is clear that
\[
\dfrac{\om_{d-1}}{2^{s+1}}B\Lr{\dfrac{s+d}2,\delta+1}\le\upsilon_{f_n,s}\le
\dfrac{\om_{d-1}}{2^{s+1}-2}B\Lr{\dfrac{s+d}2,\delta+1}.
\]
Hence $\upsilon_{f_n,0}\simeq\mc{O}(n)$ while $\upsilon_{f_n,s}\simeq\mc{O}(1)$. This shows that~\eqref{eq:banach1} is invalid for a large number $n$. This proves the second claim and completes the proof.
\end{proof}

Similar to $\ms{B}^s(\mb{R}^d)$, the space $\wh{\ms{B}}^s(\mb{R}^d)$ defined in~\eqref{eq:hat-bs} has been exploited as the target function space for neural network approximation by several authors~\cites{Ma:2017,Xu:2020,Siegel:2022,Siegel:2023}. The advantage of this space is that the Fourier transform is well-defined and the pointwise Fourier inversion is true for functions belonging to $\wh{\ms{B}}^s(\mb{R}^d)$. Unfortunately, $\wh{\ms{B}}^s(\mb{R}^d)$ is not a Banach space as shown below. 
\begin{lemma}\label{lema:hat-bs}
The space $\wh{\ms{B}}^s(\mb{R}^d),s\ge 0$ defined in~\eqref{eq:hat-bs} equipped with the norm $\upsilon_{f,0}+\upsilon_{f,s}$ is not a Banach space.
\end{lemma}

To prove Lemma~\ref{lema:hat-bs}, we recall the Barron spectrum space introduced by \textsc{Meng and Ming} in~\cite{MengMing:2022}: For $s\in\R$ and $1\le p\le 2$,
\begin{equation}\label{eq:Bsp}
\ms{B}^s_p(\mb{R}^d){:}=\set{f\in L^p(\mb{R}^d)}{\nm{f}{L^p(\mb{R}^d)}+\upsilon_{f,s}<\infty} 
\end{equation}
equipped with the norm $\nm{f}{\ms{B}^s_p(\R^d)}{:}=\nm{f}{L^p(\mb{R}^d)}+\upsilon_{f,s}$. A useful moment inequality that compares the spectral norm of different indexes has been proved in~\cite{MengMing:2022}*{Lemma 2.1}: For $1\le p\le 2$ and $-d/p<s_1<s_2$, there exists $C$ depending only on $s_1,s_2,d$ and $p$ such that
\begin{equation}\label{eq:inter}
\upsilon_{f,s_1}\le C\nm{f}{L^p(\R^d)}^{\gamma}\upsilon_{f,s_2}^{1-\gamma},
\end{equation}
where $\gamma=(s_2-s_1)/(s_2+d/p)$.  For any $\varepsilon>0$, denoting $f_\varepsilon:=f(x/\varepsilon)$ and using
\[
\upsilon_{f_{\varepsilon},s}=\varepsilon^{-s}\upsilon_{f,s},
\]
we observe that the inequality~\eqref{eq:inter}  is dilation invariant because it is invariant if we replace $f$ by $f_{\varepsilon}$. 

\begin{proof}[Proof of Lemma~\ref{lema:hat-bs}]
The authors in~\cite{MengMing:2022} have proved that $\ms{B}^s_p(\mb{R}^d)$ defined above is a Banach space. For any $f\in\ms{B}^s_1(\mb{R}^d)$, taking $s_1=0,s_2=s$ and $p=1$ in~\eqref{eq:inter}, we obtain, there exists $C$ depending only on $d$ and $s$ such that
\[
\upsilon_{f,0}\le C\nm{f}{L^1(\R^d)}^{\gamma}\upsilon_{f,s}^{1-\gamma}\le C\nm{f}{\ms{B}^s_1(\R^d)},
\]
where $\gamma=s/(s+d)$.

On the contrary, suppose that $\wh{\ms{B}}^s(\mb{R}^d)$ equipped with the norm $\upsilon_{f,0}+\upsilon_{f,s}$ is also a Banach space, then by the bounded inverse theorem and the above moment inequality~\eqref{eq:inter}, we get, there exists $C$ depending only on $s$ and $d$ such that for any $f\in\ms{B}^s_1(\mb{R}^d)$,
\begin{equation}\label{eq:banach2}
\nm{f}{L^1(\mb{R}^d)}\le C(\upsilon_{f,0}+\upsilon_{f,s}).
\end{equation}
We obtain a contradiction by the following example. 
	
For some $\delta>(d-1)/2$, we define
\[
f_n(x){:}=\Lr{\sum_{k=1}^n(1-2^{2k}\abs{\xi}^2)_+^\delta}^\vee(x).
\]
Using~\eqref{eq:phiR} and noting $f_n=\sum_{k=1}^n\phi_{2^{-k}}$, we have the explicit form of $f_n$ as
\begin{equation}\label{eq:ex2}
f_n(x)=\dfrac{\Gamma(\delta+1)}{\pi^\delta\abs{x}^{\delta+d/2}}\sum_{k=1}^n2^{k(\delta-d/2)}J_{\delta+d/2}(2^{1-k}\pi\abs{x}).
\end{equation}
Using~\eqref{eq:phiRs}, we get
\[
\upsilon_{f_n,s}=\sum_{k=1}^n\upsilon_{\phi_{2^{-k}},s}=\dfrac{1-2^{-n(s+d)}}{2^{s+d+1}-2}\omega_{d-1}
B\Lr{\dfrac{s+d}2,\delta+1},
\]
and
\[
\dfrac{\omega_{d-1}}{2^{s+d+1}}
B\Lr{\dfrac{s+d}2,\delta+1}\le\upsilon_{f_n,s}\le\dfrac{\omega_{d-1}}{2^{s+d+1}-2}
B\Lr{\dfrac{s+d}2,\delta+1}.
\]

Proceeding along the same line, we obtain
\[
\upsilon_{f_n,0}=\sum_{k=1}^n\upsilon_{\phi_{2^{-k}},0}=\dfrac{1-2^{-nd}}{2^{d+1}-2}\omega_{d-1}
B\Lr{\dfrac{d}2,\delta+1},
\]
and
\[
\dfrac{\omega_{d-1}}{2^{d+1}}
B\Lr{\dfrac{d}2,\delta+1}\le\upsilon_{f_n,0}\le\dfrac{\omega_{d-1}}{2^{d+1}-2}
B\Lr{\dfrac{d}2,\delta+1}.
\]

Hence,
\begin{equation}\label{eq:sumbd}
\upsilon_{f_n,0}+\upsilon_{f_n,s}\le\dfrac{\om_{d-1}}{2}\Lr{\dfrac{B(d/2,\delta+1)}{2^d-1}+\dfrac{B((s+d)/2,\delta+1)}{2^{s+d}-1}}.
\end{equation}

By~\eqref{eq:phiR}, a direct calculation gives
\[
\nm{\phi_R}{L^1(\mb{R}^d)}=\dfrac{\Gamma(\delta+1)}{\pi^\delta R^{\delta-d/2}}\int_{\mb{R}^d}\dfrac{\abs{J_{\delta+d/2}(2\pi\abs{x}R)}}{\abs{x}^{\delta+d/2}}\mr{d}x\\
=\dfrac{2^{\delta}\Gamma(\delta+1)}{\pi^{\delta+d/2}}\int_{\mb{R}^d}
\dfrac{\abs{J_{\delta+d/2}(\abs{x})}}{\abs{x}^{\delta+d/2}}\mr{d}x.
\]

Invoking~\cite{Grafakos:2014}*{Appendix B.6, B.7}, there exists $C$ depending on $\nu$ such that
\[\abs{J_{\nu}(x)}\le C\begin{cases}
 \abs{x}^{\nu}\qquad&\abs{x}\le 1,\\
 \abs{x}^{-1/2}\qquad&\abs{x}>1.
\end{cases}\]

We get, there exists $C$ depending only on $d$ and $\delta$ such that
\begin{align*}
\nm{\phi_R}{L^1(\mb{R}^d)}&=\dfrac{2^{\delta}\Gamma(\delta+1)}{\pi^{\delta+d/2}}\Lr{\int_{\abs{x}\le 1}\dfrac{\abs{J_{\delta+d/2}(\abs{x})}}{\abs{x}^{\delta+d/2}}\mr{d}x+\int_{\abs{x}>1}\dfrac{\abs{J_{\delta+d/2}(\abs{x})}}{\abs{x}^{\delta+d/2}}\mr{d}x}\\
&\le C\Lr{\int_{\abs{x}\le 1}\dx+\int_{\abs{x}>1}\abs{x}^{-1/2-\delta-d/2}\dx}\\
&\le C\Lr{1+\dfrac{1}{\delta-(d-1)/2}},
\end{align*}
where we have used the fact $\delta>(d-1)/2$ in the last step.  Therefore, $\nm{\phi_R}{L^1(\mb{R}^d)}$ is bounded by a constant that depends only on $\delta$ and $d$ but is independent of $R$. Moreover,
\[
\nm{f_n}{L^1(\mb{R}^d)}\le\sum_{k=1}^n\nm{\phi_{2^{-k}}}{L^1(\mb{R}^d)}\le n\nm{\phi_1}{L^1(\mb{R}^d)},
\]
and by the Hausdorff-Young inequality~\eqref{eq:hyineq},
\[
\nm{f_n}{L^1(\mb{R}^d)}\ge\nm{\wh{f}_n}{L^\infty(\mb{R}^d)}=\wh{f}_n(0)=n.
\]
This means that $\nm{f_n}{L^1(\R^d)}=\mc{O}(n)$, which together with~\eqref{eq:sumbd} immediately shows that the inequality~\eqref{eq:banach2} cannot be true for sufficiently large $n$.  Hence, we conclude that $\wh{\ms{B}}^s(\mb{R}^d)$ is not a Banach space.
\end{proof}
\subsection{Embedding relations of the spectral Barron spaces}
In this part we discuss the embedding of the spectral Barron spaces.
\begin{lemma}\label{thm:monotone}\begin{enumerate}
\item The moment inequality:  For any $0\le s_1\le s\le s_2$ satisfying $s=\alpha s_1+(1-\alpha)s_2$ with $0\le\al\le 1$, and $f\in\ms{B}^{s_1}(\mb{R}^d)$, there holds
\begin{equation}\label{eq:inter2}
	\upsilon_{f,s}\le\upsilon_{f,s_1}^\alpha\upsilon_{f,s_2}^{1-\alpha},
\end{equation}
and
\begin{equation}\label{eq:inter3}
\nm{f}{\ms{B}^s(\mb{R}^d)}\le\nm{f}{\ms{B}^{s_1}(\R^d)}^\al\nm{f}{\ms{B}^{s_2}(\R^d)}^{1-\al}.
\end{equation}

\item Let $0\le s_1\le s_2$, there holds $\ms{B}^{s_2}(\mb{R}^d)\hkto\ms{B}^{s_1}(\mb{R}^d)$ with
\begin{equation}\label{eq:monotone}
\nm{f}{\ms{B}^{s_1}(\R^d)}\le\Lr{2-\dfrac{s_1}{s_2}}\nm{f}{\ms{B}^{s_2}(\R^d)}\qquad \forall f\in\ms{B}^{s_2}(\R^d).
\end{equation}
\end{enumerate}
\end{lemma}

The embedding~\eqref{eq:monotone} has been stated in~\cite{Hormander:1963}*{Theorem 2.2.2} without tracing the embedding constant.
\begin{proof}
We start with the moment inequality~\eqref{eq:inter2} for the spectral norm. For any $0\le s_1\le s\le s_2$ with $s=\al s_1+(1-\al)s_2$, using H\"older's inequality, we obtain
\[
\upsilon_{f,s}=\int_{\R^d}\Lr{\abs{\xi}^{s_1}\abs{\wh{f}(\xi)}}^{\al}\Lr{\abs{\xi}^{s_2}\abs{\wh{f}(\xi)}}^{1-\al}\md\xi
\le\upsilon_{f,s_1}^{\al}\upsilon_{f,s_2}^{1-\al}.
\]
This gives~\eqref{eq:inter2}.

Next, for $a,b,c>0$, by Young's inequality, we have
\begin{align*}
		\dfrac{a+b^\alpha c^{1-\alpha}}{(a+b)^\alpha(a+c)^{1-\alpha}}&=\Lr{\dfrac{a}{a+b}}^\alpha\Lr{\dfrac{a}{a+c}}^{1-\alpha}+\Lr{\dfrac{b}{a+b}}^\alpha\Lr{\dfrac{c}{a+c}}^{1-\alpha}\\
		&\le\alpha\dfrac{a}{a+b}+(1-\alpha)\dfrac{a}{a+c}+\alpha\dfrac{b}{a+b}+(1-\alpha)\dfrac{c}{a+c}\\
		&=1.
\end{align*}
	This yields
	\[
		a+b^\alpha c^{1-\alpha}\le(a+b)^\alpha(a+c)^{1-\alpha}.
	\]
Let $a=\upsilon_{f,0},b=\upsilon_{f,s_1}$ and $c=\upsilon_{f,s_2}$, we obtain
\[
\nm{f}{\ms{B}^s(\R^d)}=\upsilon_{f,0}+\upsilon_{f,s}
\le\upsilon_{f,0}+\upsilon_{f,s_1}^\al\upsilon_{f,s_2}^{1-\al}\le\nm{f}{\ms{B}^{s_1}(\R^d)}^\al\nm{f}{\ms{B}^{s_2}(\R^d)}^{1-\al}.	
\]
This implies~\eqref{eq:inter3}.

Next, if we take $s_1=0$ in~\eqref{eq:inter2} and $s=(1-\alpha)s_2$ with $\alpha=1-s/s_2$, then
\[
	\nm{f}{\ms{B}^s(\mb{R}^d)}\le\upsilon_{f,0}+\upsilon_{f,0}^\alpha\upsilon_{f,s_2}^{1-\alpha}\le(1+\alpha)\upsilon_{f,0}+(1-\alpha)\upsilon_{f,s_2}\le(1+\alpha)\nm{f}{\ms{B}^{s_2}(\mb{R}^d)}.
\]
This leads to~\eqref{eq:monotone} and completes the proof.
\end{proof}

The next lemma shows that $\ms{B}^s_p(\R^d)$ is a proper subspace of $\ms{B}^s(\R^d)$.
\begin{lemma}\label{lema:ex3}
For $s\ge 0$ and $1\le p\le 2$, there holds $\ms{B}^s_p(\mb{R}^d)\hkto\ms{B}^s(\mb{R}^d)$, and the inclusion is proper in the sense that for any $1\le p<\infty$, there exists $f_p\in\ms{B}^s(\mb{R}^d)$ and $f_p\not\in L^p(\mb{R}^d)$.
\end{lemma}

\begin{proof}
It follows from the moment inequality~\eqref{eq:inter} that $\upsilon_{f,0}
\le C\nm{f}{\ms{B}^s_p(\R^d)}$. Hence 
\[
\nm{f}{\ms{B}^s(\R^d)}\le C\nm{f}{\ms{B}^s_p(\R^d)}.
\]
This implies $\ms{B}^s_p(\mb{R}^d)\hkto\ms{B}^s(\mb{R}^d)$ for any $s\ge 0$ and $1\le p\le 2$. 

It remains to prove that the inclusion is proper. Let
\[
f_p(x){:}=\Lr{\abs{\xi}^{-d/p'}\chi_{[0,1)}(\abs{\xi})}^\vee(x),
\]
where $\chi_{\Omega}(t)$ is the characteristic function on $\mb{R}$ that equals to one if $t\in\Omega$ and zero otherwise. For any $s\ge 0$, a straightforward calculation gives
\[
\upsilon_{f_p,s}=\dfrac{\omega_{d-1}}{s+d/p}.
\]
Hence, $f_p\in\ms{B}^s(\mb{R}^d)$. 

What is left is to show that $f_p\notin L^p(\R^d)$, which is based on the following explicit formula for $f_p$ shown in Appendix~\ref{apd:ex3}:
\begin{equation}\label{eq:ex3}
f_p(x)={}_1F_2(d/(2p);1+d/(2p),d/2;-\pi^2\abs{x}^2)p\nu_d,
\end{equation}
where the generalized Hypergeometric function ${}_nF_m$ is defined as follows. For nonnegative integer $n,m$ and none of the parameters $\{\beta_j\}_{j=1}^m$ is a negative integer or zero,
\[
{}_nF_m(\alpha_1,\dots,\alpha_n;\beta_1,\dots,\beta_m;x){:}=\sum_{k=0}^\infty\dfrac{\prod_{j=1}^n(\alpha_j)_k}{\prod_{j=1}^m(\beta_j)_k}\dfrac{x^k}{k!}.
\]
The generalized Hypergeometric function ${}_nF_m$ converges for all finite $x$ if $n\le m$. In particular ${}_nF_m(\alpha_1,\dots,\alpha_n;\beta_1,\dots,\beta_m;0)=1$. Hence $f_p(x)$ is finite for any $x$. Using~\cite{MathaiSaxena:1973}*{Appendix}, we obtain
\[
		{}_1F_2(\alpha;\beta,\gamma;-x^2/4)\simeq\mc{O}(\abs{x}^{\alpha-\beta-\gamma+1/2}+\abs{x}^{-2\alpha})\qquad\text{when}\quad\abs{x}\to\infty.
\]
Therefore,
\[
f_p(x)\simeq\mc{O}(\abs{x}^{-(d+1)/2}+\abs{x}^{-d/p})\qquad\text{when}\quad \abs{x}\to\infty. 
\]
This immediately implies $f_p\not\in L^p(\mb{R}^d)$.
\end{proof}
\begin{remark}
The representation~\eqref{eq:ex3} is rather complicated, we give explicit formulas for certain special cases. For $d=p=1$, using the relation~\cite[\S~6.2.1, Eq.(10)]{Luke:1969}
\[
\sin x={}_0F_1(;3/2;-x^2/4)x,
\]
we obtain
\[
f_1(x)=2{}_1F_2(1/2;3/2,1/2;-\pi^2x^2)=2{}_0F_1(;3/2;-\pi^2x^2)=\dfrac{\sin(2\pi x)}{\pi x}.
\]

When $p=2$, using the identity~\cite[\S~6.2.11, Eq. (41)]{Luke:1969}
\[
C(\sqrt{2x/\pi})=\sqrt{\dfrac{2x}{\pi}}{}_1F_2(1/4;5/4,1/2;-x^2/4)\qquad\text{when}\quad x>0,
\]
we obtain
\[
f_2(x)=4{}_1F_2(1/4;5/4,1/2;-\pi^2x^2)=\dfrac2{\sqrt{\abs{x}}}C(2\sqrt{\abs{x}}).
\]
where $C$ is the Fresnel Cosine integral given by
\[
C(x)=\int_0^x\cos(\pi t^2/2)\mr{d}t\to\dfrac12\qquad\text{when}\quad x\to\infty.
\]
\end{remark}
\subsection{Relations to some classical function spaces}
In this part, we establish the embedding between the spectral Barron space $\ms{B}^s(\mb{R}^d)$ and the Besov space, and hence we bridge $\ms{B}^s(\mb{R}^d)$ and the Sobolev space as in~\cite{MengMing:2022}. Let us begin with the definition of the Fourier-analytical based Besov space~\cite{Triebel:1983}.
\begin{definition}[Besov space]
Let $\{\varphi_j\}_{j=0}^\infty\subset\ms{S}(\mb{R}^d)$ satisfies $0\le\varphi_j\le 1$ and
\[\begin{cases}
		\text{supp}(\varphi_0)\subset\Gamma_0{:}=\set{x\in\mb{R}^d}{\abs{x}\le2},\\
		\text{supp}(\varphi_j)\subset\Gamma_j{:}=\set{x\in\mb{R}^d}{2^{j-1}\le\abs{x}\le 2^{j+1}},\qquad j=1,2,\dots.
\end{cases}
\]
For every multi-index $\alpha$, there exists a positive number $c_{\alpha}$ such that
\[
2^{j\abs{\alpha}}\abs{\nabla^{\alpha}\varphi_j(x)}\le c_{\alpha}\quad\text{for all}\quad j=0,\dots,\quad\text{for all}\quad x\in\mb{R}^d,
\]
	and
	\[
	\sum_{j=0}^\infty\varphi_j(x)=1\quad\text{for every}\quad x\in\mb{R}^d.
	\]
	Let $\alpha\in\mb{R}$ and $1\le p,q\le\infty$. Define the {\em Besov space}
	\[
		B^{\alpha}_{p,q}(\mb{R}^d){:}=\set{f\in\ms{S}'(\mb{R}^d)}{\nm{f}{B^{\alpha}_{p,q}(\mb{R}^d)}<\infty}
	\]
	equipped with the norm
	\[
		\nm{f}{B^{\alpha}_{p,q}(\mb{R}^d)}{:}=\Lr{\sum_{j=0}^\infty2^{\alpha qj}\nm{(\varphi_j\wh{f})^{\vee}}{L^p(\mb{R}^d)}^q}^{1/q}\qquad\text{when}\quad q<\infty,
	\]
	and
	\[
		\nm{f}{B_{p,\infty}^\alpha(\mb{R}^d)}{:}=\sup_{j\ge 0}2^{\alpha j}\nm{(\varphi_j\wh{f})^{\vee}}{L^p(\mb{R}^d)}.
	\]
\end{definition}

Firstly we recall some well-known facts about the embedings of  Besov spaces, which was firstly proved by Taibleson in the series of work~\cite{Taibleson:1964,Taibleson:1965,Taibleson:1966}. We retain the proof in Appendix~\ref{apd:Besov} for the readers' convenience.
\begin{lemma}\label{lema:Besov}
There holds $B_{p_1,q_1}^{\alpha_1}(\mb{R}^d)\hookrightarrow B_{p_2,q_2}^{\alpha_2}(\mb{R}^d)$ if and only if $p_1\le p_2$ and one of the following conditions holds:
	\begin{enumerate}
		\item $\alpha_1-d/p_1>\alpha_2-d/p_2$ and $q_1,q_2$ are arbitrary;
		\item $\alpha_1-d/p_1=\alpha_2-d/p_2$ and $q_1\le q_2$.
	\end{enumerate}
\end{lemma}

The main result of the embedding is:
\begin{theorem}\label{thm:Besov}
\begin{enumerate}
\item There holds 
\begin{equation}\label{eq:besov}
B_{2,1}^{s+d/2}(\mb{R}^d)\hkto\ms{B}^s(\mb{R}^d)\hkto B_{\infty,1}^s(\mb{R}^d)
\end{equation}
with
\begin{equation}\label{eq:thmBesovC}
   2^{-s}\nm{f}{B_{\infty,1}^s(\mb{R}^d)}\le\nm{f}{\ms{B}^s(\mb{R}^d)}\le 2^{s+1+d/2}\sqrt{\nu_d}\nm{f}{B_{2,1}^{s+d/2}(\mb{R}^d)}.
\end{equation}
\item The above embedding is optimal in the sense that $B^\alpha_{p,q}(\mb R^d)\hkto\ms B^s(\mb R^d)$ if and only if $B^\alpha_{p,q}(\mb R^d)\hkto B^{s+d/2}_{2,1}(\mb R^d)$, and $\ms B^s(\mb R^d)\hkto B^\alpha_{p,q}(\mb R^d)$ if and only if $B^s_{\infty,1}(\mb R^d)\hkto B^\alpha_{p,q}(\mb R^d)$.
	\end{enumerate}
\end{theorem}

\begin{remark}
By tracking the embedding constant in~\eqref{eq:thmBesovC}, we observe that it is indenpendent of the dimension $d$ because
\[
2^{s+1+d/2}\sqrt{\nu_d}\sim\dfrac{2^{s+1}}{(d\pi)^{1/4}}\lr{\dfrac{8\pi e}{d}}^{d/4}\to 0\qquad\text{as}\quad d\to\infty.
\]
This suggests that the embedding is effective in high dimension. It also indicates that the spectral Barron space is more suitable than the Besov space as a target function space for neural network approximation because
\[
\nm{f}{\ms{B}(\mb{R}^d)}\le Cd^{-(d+1)/4}\nm{f}{B^{s+d/2}_{2,1}(\mb{R}^d)}\qquad\text{for any\quad}f\in B^{s+d/2}_{2,1}(\mb{R}^d)
\]
with a universal constant $C$.
\end{remark}

\begin{proof}[Proof of Theorem~\ref{thm:Besov}]
	To prove (1), firstly, for any $f\in\ms{B}^s(\mb{R}^d)$,
	\begin{align*}
		\nm{f}{\ms{B}^s(\mb{R}^d)}=&\sum_{j=0}^\infty\int_{\mb{R}^d}(1+\abs{\xi}^s)\varphi_j(\xi)\abs{\wh{f}(\xi)}\mr{d}\xi\\
		\le&\sum_{j=0}^\infty\Lr{\int_{\text{supp\;}\varphi_j}(1+\abs{\xi}^s)^2\mr{d}\xi}^{1/2}\nm{\varphi_j\wh{f}}{L^2(\mb{R}^d)}.
	\end{align*}
A direct calculation gives: for $j=0,1,\dots$,
\begin{align*}
\int_{\text{supp\;}\varphi_j}(1+\abs{\xi}^s)^2\mr{d}\xi
&\le\int_{0\le\abs{\xi}\le 2^{j+1}}(1+\abs{\xi}^s)^2\mr{d}\xi\\
&=\omega_{d-1}\int_0^{2^{j+1}}(1+r^s)^2r^{d-1}\mr{d}\,r\\
&\le 2\omega_{d-1}\int_0^{2^{j+1}}(1+r^{2s})r^{d-1}\mr{d}\,r\\
&\le2\omega_{d-1}\Lr{\dfrac{2^{(j+1)d}}{d}+\dfrac{2^{(j+1)(2s+d)}}{2s+d}}\\
&\le 4\nu_d2^{(j+1)(2s+d)}.
\end{align*}
Using the Plancherel's theorem, we get
\begin{align*}
\nm{f}{\ms{B}^s(\mb{R}^d)}&\le 2^{s+1+d/2}\sqrt{\nu_d}
\sum_{j=0}^\infty2^{j(s+d/2)}\nm{\varphi_j\wh{f}}{L^2(\mb{R}^d)}\\
&=2^{s+1+d/2}\sqrt{\nu_d}
\sum_{j=0}^\infty2^{j(s+d/2)}\nm{(\varphi_j\wh{f})^\vee}{L^2(\mb{R}^d)}\\
&= 2^{s+1+d/2}\sqrt{\nu_d}\nm{f}{B_{2,1}^{s+d/2}(\mb{R}^d)}.
\end{align*}

Next, for any $f\in \ms{B}^s(\mb{R}^d)$, by Lemma~\ref{lema:fourinv}, we have $\varphi_j\wh{f}\in L^1(\mb{R}^d)$, using the Hausdorff-Young inequality~\eqref{eq:hyineq}, we obtain
\begin{align*}
		\nm{f}{B_{\infty,1}^s(\mb{R}^d)}=&\sum_{j=0}^\infty2^{sj}\nm{(\varphi_j\wh{f})^\vee}{L^\infty(\mb{R}^d)}\le\sum_{j=0}^\infty2^{sj}\nm{\varphi_j\wh{f}}{L^1(\mb{R}^d)}\\
		\le&\nm{\varphi_0\wh{f}}{L^1(\mb{R}^d)}+2^s\sum_{j=1}^\infty\int_{\mb{R}^d}\varphi_j(\xi)\abs{\xi}^s\abs{\wh{f}(\xi)}\mr{d}\xi\\
\le&\upsilon_{f,0}+2^s\upsilon_{f,s}.
\end{align*}
Therefore, $\nm{f}{B_{\infty,1}^s(\mb{R}^d)}\le2^s\nm{f}{\ms{B}^s(\mb{R}^d)}$. This proves~\eqref{eq:besov} with
\[
2^{-s}\nm{f}{B_{\infty,1}^s(\mb{R}^d)}\le\nm{f}{\ms{B}^s(\mb{R}^d)}\le 2^{s+1+d/2}\sqrt{\nu_d}\nm{f}{B_{2,1}^{s+d/2}(\mb{R}^d)}.
\]

It remains to show that the embedding~\eqref{eq:besov} is optimal. The ``if''-part is apparent, it suffices to show the ``only if''-part. On the one hand, suppose
\(
B_{p,q}^\alpha(\mb{R}^d)\hookrightarrow\ms{B}^s(\mb{R}^d),
\) one would have $B_{p,q}^\alpha(\mb{R}^d)\hookrightarrow B_{\infty,1}^s(\mb{R}^d)$. Using Lemma~\ref{lema:Besov}, we conclude that the triple $(p,q,\alpha)$ must be in the set
\begin{equation}\label{eq:thmBesovL}
    S_1{:}=\set{p,q\in[1,\infty],\alpha\in\mb R}{\alpha>s+d/p,\text{ or }\alpha=s+d/p\text{ and }q=1}.
\end{equation}
We split this set into three subsets. 

First, if $(p,q,\alpha)\in S_1$ with $1\le p\le 2$, then by Lemma~\ref{lema:Besov} we obtain $B_{p,q}^\alpha(\mb{R}^d)\hkto B_{2,1}^{s+d/2}(\mb R^d)$. This proves the first assertion. 

Second, if $(p,q,\alpha)\in S_1$ with $2< p<\infty$, then we shall exploit an example adopted from~\cite[Ch. 5, Ex. 9]{Lieb:2001} to show that $B_{p,q}^\alpha(\mb{R}^d)\not\hkto\ms{B}^s(\mb{R}^d)$. Let 
\[
\psi_n(x)=(1+in)^{-d/2}e^{-\pi\abs{x}^2/(1+in)}.
\]
A direct calculation gives $\wh{\psi}_n(\xi)=e^{-\pi(1+in)\abs{\xi}^2}$. Hence 
$\abs{\wh{\psi}_n(\xi)}=e^{-\pi\abs{\xi}^2}\in\ms{S}(\mb{R}^d)$ and
\[
\nm{\psi_n}{\ms{B}^s(\mb{R}^d)}=1+\dfrac{\Gamma((s+d)/2)}{\Gamma(d/2)\pi^{s/2}},
\]
which is independent of $n$. We shall show in Appendix~\ref{apd:ex5} that there exists $C$ independent of $n$ such that
\begin{equation}\label{eq:ex5}
\nm{\psi_n}{B_{p,q}^\alpha(\mb{R}^d)}\le C(1+n^2)^{-d(p-2)/(4p)}
\end{equation}
when $2\le p<\infty$ and $\alpha>0$. Therefore, $\nm{\psi_n}{B_{p,q}^\alpha(\mb{R}^d)}\to 0$ when $p>2$ and $n\to\infty$. This proves $B_{p,q}^\alpha(\mb{R}^d)\not\hkto\ms{B}^s(\mb{R}^d)$.

Finally, if $(p,q,\alpha)\in S_1$ with $p=\infty$. Note that $\varphi_j(x)=\delta_{0j}$ when  $\abs{x}<1$.
A direct calculation gives
\[
\nm{1}{B^{\alpha}_{\infty,q}(\mb R^d)}=\nm{(\varphi_0\wh{1})^\vee}{L^\infty(\mb R^d)}=\nm{1}{L^\infty(\mb R^d)}=1,
\]
while it is clear that the constant function $1\not\in\ms B^s(\mb R^d)$. This implies $B^\alpha_{\infty,q}(\mb R^d)\not\hkto\ms B^s(\mb R^d)$.

Summing up the above three cases, we conclude the first part of the assertion (2).

On the other hand, suppose
\(
    \ms B^s(\mb R^d)\hkto B^\alpha_{p,q}(\mb R^d),
\)
one would have $B^{s+d/2}_{2,1}(\mb R^d)\hkto B^\alpha_{p,q}(\mb R^d)$. Invoking Lemma~\ref{lema:Besov} again, we conclude that the index $(p,q,\alpha)$ must be in the set
\begin{equation}\label{eq:thmBesovR}
   S_2{:}=\set{p\in [2,\infty],q\in[1,\infty],\alpha\in\mb R}{\alpha\le s+d/p}.
\end{equation}
Again we split this set into two subsets.

First, if $(p,q,\alpha)\in S_2$ with $2\le p<\infty$, then we have proved that $f_p$ constructed in Lemma~\ref{lema:ex3} satisfying $f_p\in\ms B^s(\mb R^d)$ while $f_p\not\in L^p(\mb R^d)$. A direct calculation gives
\[
\nm{f_p}{B^\alpha_{p,q}(\mb R^d)}=\nm{(\varphi_0\wh{f}_p)^\vee}{L^p(\mb R^d)}=\nm{f_p}{L^p(\mb R^d)}.
\]
Hence $f_p\not\in B^\alpha_{p,q}(\mb R^d)$ and it yields $\ms B^s(\mb R^d)\not\hkto B^\alpha_{p,q}(\mb R^d)$.

Second, if $(p,q,\alpha)\in S_2$ with $p=\infty$. As expected, Lemma~\ref{lema:Besov} yields that $B^s_{\infty,1}(\mb R^d)\hkto B^\alpha_{\infty,q}(\mb R^d)$. Hence we prove the second part of the assertion (2). Therefore, we conclude that the embedding~\eqref{eq:besov} is optimal.
\end{proof}

As a consequence of Theorem~\ref{thm:Besov} and Lemma~\ref{lema:ex3}, we establish the embedding between the spectral Barron space and the Sobolev spaces. 
\begin{definition}[Fractional Sobolev space]
	Let $1\le p<\infty$ and non-integer $\alpha>0$, then the fractional Sobolev space
	\[
		W^\alpha_p(\mb{R}^d){:}=\set{f\in W^{\lfloor\alpha\rfloor}_p(\mb{R}^d)}{\iint_{\mb{R}^d\times\mb{R}^d}\dfrac{\abs{\nabla^{\lfloor\alpha\rfloor}f(x)-\nabla^{\lfloor\alpha\rfloor}f(y)}^p}{\abs{x-y}^{d+(\alpha-\lfloor\alpha\rfloor)p}}\mr{d}x\mr{d}y<\infty}
	\]
	equipped with the norm
	\[
		\nm{f}{W^\alpha_p(\mb{R}^d)}{:}=\nm{f}{W^{\lfloor\alpha\rfloor}_p(\mb{R}^d)}+\Lr{\iint_{\mb{R}^d\times\mb{R}^d}\dfrac{\abs{\nabla^{\lfloor\alpha\rfloor}f(x)-\nabla^{\lfloor\alpha\rfloor}f(y)}^p}{\abs{x-y}^{d+(\alpha-\lfloor\alpha\rfloor)p}}\mr{d}x\mr{d}y}^{1/p}.
	\]
\end{definition}

It is a straightforward corollary of the following  embedding relation between the Sobolev space and $\ms{B}^s_p(\R^d)$.
\begin{lemma}[\cite{MengMing:2022}*{Theorem 4.3}]\label{lema:BarronSobolev}
\begin{enumerate}
\item If $1\le p\le 2$ and $\al>s+d/p>0$, then
\[
W^{\al}_p(\R^d)\hkto\ms{B}^s_p(\R^d).
\]
\item If $s>-d$ is not an integer or $s>-d$ is an integer and $d\ge 2$, then
\[
W^{s+d}_1(\R^d)\hkto\ms{B}^s_1(\R^d).
\]
\end{enumerate}\end{lemma}

It follows from the above lemma and Lemma~\ref{lema:ex3} that
\begin{coro}\label{coro:Sobolev}
\begin{enumerate}
\item If $1\le p\le 2$ and $\al>s+d/p$, there holds 
\begin{equation}\label{eq:sobolevembed}
W^{\al}_p(\mb{R}^d)\hkto\ms{B}^s(\mb{R}^d)\hkto C^s(\mb{R}^d).
\end{equation}
	
\item If $s$ is not an integer or $s$ is an integer and $d\ge 2$, then 
\[
W^{s+d}_1(\mb{R}^d)\hkto\ms{B}^s(\mb{R}^d).
\] 
\end{enumerate}
\end{coro}

The first embedding with $p=2$ and $s=1$ was hidden in~\cite{Barron:1993}*{~\S~II, Para. 7;~\S~IX, 15}.
\begin{proof}
By Lemma~\ref{lema:BarronSobolev} and Lemma~\ref{lema:ex3}, when $\al>s+d/p$ and $1\le p\le 2$, we have
\[
W^{\al}_p(\mb{R}^d)\hkto\ms{B}^s_p(\mb{R}^d)\hkto\ms{B}^s(\R^d).
\]

When $s$ is not an integer or $s$ is an integer and $d\ge 2$, there holds
\[
W^{s+d}_1(\mb{R}^d)\hkto\ms{B}^s_1(\mb{R}^d)\hkto\ms{B}^s(\R^d).
\] 

It remains to prove the right-hand side of~\eqref{eq:sobolevembed}. Using Theorem~\ref{thm:Besov}, 
\[
\ms{B}^s(\mb{R}^d)\hkto B^s_{\infty,1}(\mb{R}^d)\hookrightarrow C^s(\mb{R}^d)
\]
due to Theorem~\ref{thm:Besov}, Lemma~\ref{lema:Besov} and~\cite{Triebel:1983}*{\S~2.3.5, Eq. (1); \S~2.5.7, Eq. (2), (9), (11)}.
\end{proof}
\section{Application to deep neural network approximation}\label{sec:appro}
The embedding results proved in Theorem~\ref{thm:Besov} and Corollary~\ref{coro:Sobolev} indicate that $s$ is a smoothness index. Consequently, we are interested in exploring the approximate rate when $s$ is small with $\ms{B}^s$ as the target function space. To facilitate our analysis, we shall focus on the hypercube $\Omega{:}=[0,1]^d$, and the spectral norm for function $f$ with respect to $\Omega$ is
\[
\upsilon_{f,s,\Omega}=\inf_{Ef|_\Omega=f}\int_{\mb{R}^d}|\xi|_1^s\abs{\wh{Ef}(\xi)}\mr{d}\xi.
\]
Here we replace $\abs{\xi}$ by $|\xi|_1$ in the definition of $\upsilon_{f,s,\Omega}$, the latter seems more natural for studying the approximation over the hypercube as suggested by~\cite{Barron:1993}*{\S~V}.
\begin{definition}
	A sigmoidal function is a bounded function $\sigma:\mb{R}\mapsto\mb{R}$ such that
	\[
		\lim_{t\to-\infty}\sigma(t)=0,\qquad \lim_{t\to+\infty}\sigma(t)=1.
	\]
	For example, the Heaviside function $\chi_{[0,\infty)}$ is a sigmoidal function.
\end{definition}

A classical idea for the approximation error of neural networks with  sigmoidal activation functions $\sigma$ is to use the Heaviside function $\chi_{[0,\infty)}$ as a transition. \textsc{Caragea et. al.}~\cite{Caragea:2023} pointed out that the gap between sigmoidal function $\sigma$ and the Heaviside function $\chi_{[0,\infty)}$ cannot be dismissed in $L^\infty(\Omega)$. While this gap does not exist in $L^2(\Omega)$.
\begin{lemma}\label{lema:sigmoidal}
	For fixed $\omega\in\mb{R}^d\backslash\{0\}$ and $b\in\mb{R}$,
	\[
		\lim_{\tau\to\infty}\nm{\sigma(\tau(\omega\cdot x+b))-\chi_{[0,\infty)}(\omega\cdot x+b)}{L^2(\Omega)}=0.
	\] 
\end{lemma}

\begin{proof}
	Note that
	\[
		\lim_{t\to\pm\infty}\abs{\sigma(t)-\chi_{[0,\infty)}(t)}=0.
	\]
	We divide the cube $\Omega$ into $\Omega_1{:}=\set{x\in\Omega}{\abs{\tau(\omega\cdot x+b)}<\delta}$ and $\Omega_2{:}=\Omega\setminus\Omega_1$. With proper choice of $\delta>0$ and $\tau>0$ large enough, we can obtain that the $L^2$-distance between $\sigma(\tau(\omega\cdot x+b))$ and $\chi_{[0,\infty)}(\omega\cdot x+b)$ is arbitrarily small.
\end{proof}
For a shallow neural network, the following lemma in~\cite{Barron:1993} is proved for the real-valued function, while it is straightforward to extend the proof to the complex-valued function.
\begin{lemma}[\cite{Barron:1993}*{Theorem 1}]\label{lema:low}
	Let $f\in\ms{B}^1(\mb{R}^d)$, there exists
	\begin{equation}\label{eq:shallow}
		f_N(x)=\sum_{i=1}^Nc_i\sigma(\omega_i\cdot x+b_i)
	\end{equation}
	with $\omega_i\in\mb{R}^d,b_i\in\mb{R}$ and $c_i\in\mb{C}$ such that
	\[
		\nm{f-f_N}{L^2(\Omega)}\le\dfrac{2\upsilon_{f,1,\Omega}}{\sqrt{N}}.
	\]
\end{lemma}

In this part, we shall show the approximation error for the deep neural network. We use the $(L,N)$-network to describe a neural network with $L$ hidden layers and at most $N$ units per layer. Here $L$ denotes the number of hidden layers, e.g., the shallow neural network, expressed as~\eqref{eq:shallow}, is an $(1,N)$-network.
\begin{definition}[$(L,N)$-network]
An $(L,N)$-network represents a neural network with $L$ hidden layers and at most $N$ units per layer. The activation functions of the first $L-1$ layers are all ReLU and the activation function of the last layer is the sigmoidal function. The connection weights between the input layer and the hidden layer, and between the hidden layer and the hidden layer are all real numbers. The connection weights between the last hidden layer and the output layer are complex numbers.
\end{definition}

There is relatively little work on the approximation rate of deep neural networks that utilize the spectral Barron space as the target space. For deep ReLU networks,~\cite{BreslerNagaraj:2020} has proven approximation results of $(sL/2)$-order. The main contribution of this section is to improve this result to an $sL$-order approximation, at the cost of introducing $\upsilon_{f,s,\Omega}$ in the estimate; cf. Theorem~\ref{thm:deep}.

\begin{theorem}\label{thm:deep}
	Let the positive integer $L$ and $f\in\ms{B}^s(\mb{R}^d)$ with $0<sL\le 1/2$. For any positive integer $N$ there exists an $(L,N+2)$-network $f_N$ such that
	\begin{equation}\label{eq:thm2}
		\nm{f-f_N}{L^2(\Omega)}\le\dfrac{29\upsilon_{f,s,\Omega}}{N^{sL}}.
	\end{equation}
 Moreover, if $f$ is a real-valued function, then the connection weights in $f_N$ are all real.
\end{theorem}

 As far as we are aware, the above theorem represents the state-of-the-art in the literature to date. For shallow neural network $L=1$, the authors in~\cite{MengMing:2022} established a $1/2$-order convergence with target function space $\ms{B}^{1/2}_2(\mb{R}^d)$, which is a subspace of $\ms{B}^{1/2}(\R^d)$. Their estimate depends on the dimension with a factor $d^{1/4}$. The upper bound in~\cite{Siegel:2022} depends on $\upsilon_{f,0}+\upsilon_{f,1/2}$, whereas~\eqref{eq:ex1} demonstrates that $\upsilon_{f,0}$ can be much larger than $\upsilon_{f,s}$ for certain functions in $\ms{B}^s(\mb{R}^d)$, and the estimate depends on the dimension exponentially. In contrast to these two results, the upper bound in Theorem~\ref{thm:deep} relies solely on $\upsilon_{f,s,\Omega}$, and is independent of the dimension.

For deep neural network, a similar result for the ReLU activation function has been presented in~\cite{BreslerNagaraj:2020} with a $(sL/2)$-order approximation. Compared to this result, our estimate exhibits a higher level of optimality. At first glance, our result might seem to contradict~\cite{BreslerNagaraj:2020}*{Theorem 2}. In reality, this is not the case because the upper bound in that reference is $\sqrt{\upsilon_{f,0}\upsilon_{f,s}}+\upsilon_{f,0}$, which requires $f\in\ms{B}^s(\mb{R}^d)$, but is typically smaller than $\nm{f}{\ms{B}^s(\mb{R}^d)}$ for oscillatory functions; cf. Lemma~\ref{lema:decay}.

In what follows, we make some preparations to prove Theorem~\ref{thm:deep}. The analysis in this part owns the most to~\cite{BreslerNagaraj:2020} with certain improvements that will be detailed later on. For any function $g$ defined on $[0,1]$ and it is symmetric about $x=1/2$, We use the notation $g_{,n}$ to denote the function $g$ in the $[0,1]$ interval of the period repeated $n$ times, i.e.,
\begin{equation}\label{eq:gn}
	g_{,n}(t)=g(nt-j),\quad j=0,\dots,n-1,\quad 0\le nt-j\le 1.
\end{equation}
Define
\[
	\beta(t)=\text{ReLU}(2t)-2\text{ReLU}(2t-1)+\text{ReLU}(2t-2)=\begin{cases}
		2t,&0\le t\le 1/2,\\
		2-2t,&1/2\le t\le 1,\\
		0,&\text{otherwise}.
	\end{cases}
\]
By definition~\eqref{eq:gn}, $\beta_{,n}$ represents a triangle function with $n$ peaks and can be represented by $3n$ ReLUs:
\[
	\beta_{,n}(t)=\sum_{j=0}^{n-1}\beta(nt-j),\quad 0\le t\le 1.
\]
\begin{lemma}\label{lema:comp}
	Let $g$ be a function defined on $[0,1]$ and symmetric about $x=1/2$, then $g_{,n_2}\circ\beta_{,n_1}=g_{,2n_1n_2}$ on $[0,1]$.
\end{lemma}

The above lemma is a rigorous statement of~\cite{Telgarsky:2016}*{Proposition 5.1}. A key example is $\cos(2\pi n_2\beta_{,n_1}(t))=\cos(4\pi n_1n_2t)$ when $t\in [0, 1]$. A geometrical explanation may be founded in ~\cite{Bolcskei:2021}*{Figure 3}. We postpone the rigorous proof in Appendix~\ref{apd:comp}.

For $r\in(0,1)$, we define 
\[
	\alpha(t,r)=\chi_{[0,\infty)}(t-r/2)-\chi_{[0,\infty)}(t-(1-r)/2)=\begin{cases}
		\chi_{[r/2,(1-r)/2]}(t), & 0< r\le 1/2,\\
		-\chi_{[(1-r)/2,r/2]}(t), & 1/2\le r< 1,
	\end{cases}
\]
then supp$(\alpha(\cdot,r))\subset(0,1/2)$ and $\alpha(t,r)$ is symmetric about $t=1/4$. Define
\[
	\gamma(t,r)=\alpha(t+1/4,r)-\alpha(t-1/4,r)+\alpha(t-3/4,r).
\]
Then $\gamma(t,r)$ is symmetric about $t=1/2$ because
\begin{align*}
	\gamma(1-t,r)=&\alpha(5/4-t,r)-\alpha(3/4-t,r)+\alpha(1/4-t,r)\\
	=&\alpha(t-3/4,r)-\alpha(t-1/4,r)+\alpha(t+1/4,r)=\gamma(t,r).
\end{align*}
By definition~\eqref{eq:gn}, $\gamma_{,n}(\cdot,r)$ is well defined on $[0,1]$ and
\begin{align*}
	\gamma_{,n}(t,r)=&\begin{cases}
		\alpha(nt-j+1/4,r), & 0\le nt-j \le 1/4, \\
		-\alpha(nt-j-1/4,r), & 1/4\le nt-j\le 3/4,\\
		\alpha(nt-j-3/4,r), & 3/4\le nt-j\le 1,
	\end{cases}&& j=0,\dots,n-1\\
	=&\begin{cases}
		\alpha(nt-j+1/4,r), & -1/4\le nt-j\le 1/4,\\
		-\alpha(nt-j-1/4,r), & 1/4\le nt-j\le 3/4,\\
	\end{cases} && j=0,\dots,n.
\end{align*}
And $\gamma_{,n}(\cdot,r)$ on $[0,1]$ can be represents by $4n$ Heaviside function $\chi_{[0,\infty)}$ due to $\alpha(nt+1/4,r),\alpha(nt-n+1/4,r)$ only need one Heaviside function each:
\[
	\gamma_{,n}(t,r)=\sum_{j=0}^n\alpha(nt-j+1/4,r)-\sum_{j=0}^{n-1}\alpha(nt-j-1/4,r).
\] 

A direct consequence of the above construction is
\begin{lemma}\label{lema:cos}
For $t\in[0,1]$, there holds
	\begin{equation}\label{eq:cos}
		\dfrac\pi{2}\int_0^1\cos(\pi r)\gamma_{,n}(t,r)\mr{d}r=\cos(2\pi nt).
	\end{equation}
\end{lemma}

\begin{proof}
	For any $t\in[0,1/2]$, a direct calculation gives
	\[
		\dfrac\pi{2}\int_0^1\cos(\pi r)\alpha(t,r)\mr{d}r=\pi\int_0^{2t}\cos(\pi r)\mr{d}r=\sin(2\pi t).
	\]
	Fix a $t\in[0,1]$. If there exists an integer $j$ satisfying $0\le j\le n$ and $-1/4\le nt-j\le 1/4$, then $\gamma_{,n}(t,r)=\alpha(nt-j+1/4,r)$ and
	\[
		\dfrac\pi{2}\int_0^1\cos(\pi r)\alpha(nt-j+1/4,r)\mr{d}r=\sin(2\pi(nt-j+1/4))=\cos(2\pi nt).
	\]
	Otherwise there exists an integer $j$ satisfying $0\le j\le n-1$ and $1/4\le nt-j\le 3/4$. Then $\gamma_{,n}(t,r)=-\alpha(nt-j-1/4,r)$ and
	\[
		-\dfrac\pi{2}\int_0^1\cos(\pi r)\alpha(nt-j-1/4,r)\mr{d}r=-\sin(2\pi(nt-j-1/4))=\cos(2\pi nt).
	\]
	This completes the proof of~\eqref{eq:cos}.
\end{proof}

The following Lemma is the key point to prove Theorem~\ref{thm:deep}, which follows the framework of~\cites{BreslerNagaraj:2020}, while we achieve a higher order convergence rate and the constant is dimension-free for high-frequency function.
\begin{lemma}\label{lema:deep}
Let the positive integer $L$ and $f\in\ms{B}^s(\mb{R}^d)$ with $0<sL\le 1/2$ and $\mr{supp\;}\wh{f}\subset\set{\xi\in\mb{R}^d}{|\xi|_1\ge 1}$. For any positive integer $N$ there exists an $(L,N)$-network $f_N$ such that
	\begin{equation}\label{eq:L2deep}
		\nm{f-f_N}{L^2(\Omega)}\le\dfrac{22\upsilon_{f,s,\Omega}}{N^{sL}}.
	\end{equation}
\end{lemma}

\begin{proof}
	By Lemma~\ref{lema:fourinv}, for $f\in\ms{B}^s(\mb{R}^d)$, assume $f$ is real-valued, then
	\[
		f(x)=\int_{\mb{R}^d}\wh{f}(\xi)e^{2\pi i\xi\cdot x}\mr{d}\xi=\int_{\mb{R}^d}\abs{\wh{f}(\xi)}\cos(2\pi(\xi\cdot x+\theta(\xi)))\mr{d}\xi,
	\]
	with proper choice $\theta(\xi)$ such that $0\le\xi\cdot x+\theta(\xi)\le|\xi|_1+1$. For fixed $\xi$, choose $n_\xi=2^{L-1}\lceil(|\xi|_1+1)^{1/L}\rceil^L$ and $t_\xi(x)=(\xi\cdot x+\theta(\xi))/n_\xi$, then $0\le t_\xi(x)\le 1$ and by Lemma~\ref{lema:cos}, we reshape $f(x)$ as
	\[
		f(x)=\int_{\mb{R}^d}\abs{\wh{f}(\xi)}\cos(2\pi n_\xi t_\xi(x))\mr{d}\xi=\dfrac\pi{2}\int_{\mb{R}^d}\abs{\wh{f}(\xi)}\mr{d}\xi\int_0^1\cos(\pi r)\gamma_{,n_\xi}(t_\xi(x),r)\mr{d}r.
	\]
	Define the probability measure
	\begin{equation}\label{eq:prob}
		\mu(\mr{d}\xi,\mr{d}r)=\dfrac1Q|\xi|_1^{-s}\abs{\wh{f}(\xi)}\chi_{(0,1)}(r)\mr{d}\xi\mr{d}r,
\end{equation}
    where $Q$ is the normalized factor that
	\[
		Q=\int_{\mb{R}^d}|\xi|_1^{-s}\abs{\wh{f}(\xi)}\mr{d}\xi\int_0^1\mr{d}r\le\upsilon_{f,s,\Omega}.
	\]
	Therefore $f(x)=\mb{E}_{(\xi,r)\sim\mu}F(x,\xi,r)$ with
	\[
		F(x,\xi,r)=\dfrac{\pi Q}{2}|\xi|_1^s\cos(\pi r)\gamma_{,n_\xi}(t_\xi(x),r).
	\]
	If $\{\xi_i,r_i\}_{i=1}^m$ is an i.i.d. sequence of random samples from $\mu$, and
	\[
		\Tilde{f}=\dfrac1m\sum_{i=1}^mF(x,\xi_i,r_i),
	\]
	then using Fubini's theorem, we obtain
\begin{align*}
		\mb{E}_{(\xi_i,r_i)\sim\mu}\nm{f-\Tilde{f}}{L^2(\Omega)}^2=&\int_\Omega\mb{E}_{(\xi_i,r_i)\sim\mu}\abs{\mb{E}_{(\xi,r)\sim\mu}F(x,\xi,r)-\Tilde{f}(x)}^2\mr{d}x\\
		=&\dfrac1m\int_\Omega\mathrm{Var}_{(\xi,r)\sim\mu}F(x,\xi,r)\mr{d}x\\
		\le&\dfrac1m\mb{E}_{(\xi,r)\sim\mu}\nm{F(\cdot,\xi,r)}{L^\infty(\Omega)}^2.
\end{align*}
Note that
	\[
		\nm{F(\cdot,\xi,r)}{L^\infty(\Omega)}\le\dfrac{\pi Q}{2}|\xi|_1^s,
	\]
we obtain
	\[
		\mb{E}_{(\xi_i,r_i)\sim\mu}\nm{f-\Tilde{f}}{L^2(\Omega)}^2\le\dfrac1m\mb{E}_{(\xi,r)\sim\mu}\nm{F(\cdot,\xi,r)}{L^\infty(\Omega)}^2\le\dfrac{\pi^2Q\upsilon_{f,s,\Omega}}{4m}.
\]
By Markov's inequality, with probability at least $(1+\varepsilon)/(2+\varepsilon)$, for some $\varepsilon>0$ to be chosen later on, we obtain
\begin{equation}\label{eq:L2err}
    \nm{f-\Tilde{f}}{L^2(\Omega)}^2\le\dfrac{(2+\varepsilon)\pi^2Q\upsilon_{f,s,\Omega}}{4m}
\end{equation}
 
	It remains to calculate the number of units in each layer. For each $\gamma_{,n_\xi}(t_\xi(x),r)$, choose $n_1=\dots=n_L=\lceil(|\xi|_1+1)^{1/L}\rceil$, then $n_\xi=2^{L-1}n_1\dots n_L$, and by Lemma~\ref{lema:comp}, $\gamma_{,n_\xi}(\cdot,r)=\gamma_{,n_L}(\cdot,r)\circ\beta_{,n_{L-1}}\circ\dots\circ\beta_{,n_1}$ on $[0,1]$. Lemma~\ref{lema:sigmoidal} shows the Heaviside function $\chi_{[0,\infty)}$ can be approximated by $\sigma$ with at most
	\[
		\max\{3n_1,\dots,3n_{L-1},4n_L\}\le 4\lceil(|\xi|_1+1)^{1/L}\rceil\le 12|\xi|_1^{1/L}
	\]
units in each layer to represent $\gamma_{,n_\xi}(t_\xi(x),r)$. Denote $N$ the total number of units in each layer, then $N\le 12\sum_{i=1}^m|\xi_i|_1^{1/L}$ and
	\[
		\mb{E}_{(\xi_i,r_i)\sim\mu}N^{2sL}\le 12\sum_{i=1}^m\mb{E}_{(\xi_i,r_i)\sim\mu}|\xi_i|_1^{2s}\le\dfrac{12m\upsilon_{f,s,\Omega}}{Q}.
	\]
Invoking Markov inequality again, with probability at least $(1+\varepsilon)/(2+\varepsilon)$, we obtain
\begin{equation}\label{eq:num}
    \dfrac{Q}{m}\le\dfrac{12(2+\varepsilon)\upsilon_{f,s,\Omega}}{N^{2sL}}.
\end{equation}
Combining~\eqref{eq:L2err} and~\eqref{eq:num}, with probability at least $\varepsilon/(2+\varepsilon)$, there exists an $(L,N)$-network $f_N$ such that
	\[
		\nm{f-f_N}{L^2(\Omega)}\le\dfrac{\sqrt3(2+\varepsilon)\pi\upsilon_{f,s,\Omega}}{N^{sL}}\le\dfrac{11\upsilon_{f,s,\Omega}}{N^{sL}},
	\]
 with proper choice of $\varepsilon$ in the last step. Finally, if $f$ is complex-valued, we approximate the real and imaginary parts of the function separately to obtain~\eqref{eq:L2deep}.
\end{proof}

\begin{remark}
We assume $\mathrm{supp}\;\wh{f}\subset\set{\xi\in\mb{R}^d}{|\xi|_1\ge 1}$ in Lemma~\ref{lema:deep} because we want to obtain an upper bound depending only on $\upsilon_{f,s,\Omega}$. If we give up this condition, then the upper bound in Theorem~\ref{thm:deep} changes to $C\nm{f}{\ms{B}^s(\mb{R}^d)}/N^{sL}$ for some dimension-free constant $C$. The proof is essentially the same provided that the probability measure~\eqref{eq:prob} is replaced by
\[
\mu(\mr{d}\xi,\mr{d}r)=\dfrac1Q(1+|\xi|_1)^{-s}\abs{\wh{f}(\xi)}\chi_{(0,1)}(r)\mr{d}\xi\mr{d}r,
\]
We leave it to the interested reader.
\end{remark}

\begin{proof}[Proof of Theorem~\ref{thm:deep}]
We write $f=f_1+f_2$ with
	\[
		f_1(x)=\int_{|\xi|_1<1}\wh{f}(\xi)e^{2\pi i\xi\cdot x}\mr{d}\xi,\qquad f_2(x)=\int_{|\xi|_1\ge 1}\wh{f}(\xi)e^{2\pi i\xi\cdot x}\mr{d}\xi.
	\]
	Then $\upsilon_{f_1,1,\Omega}\le\upsilon_{f,s,\Omega}$ and $\upsilon_{f_2,s,\Omega}\le\upsilon_{f,s,\Omega}$ because
	\[
		\wh{f}_1(\xi)=\wh{f}(\xi)\chi_{[0,1)}(|\xi|_1)\qquad\text{and}\qquad \wh{f}_2(\xi)=\wh{f}(\xi)\chi_{[1,\infty)}(|\xi|_1).
	\]
We approximate $f_1$ with an $(L,n_1)$-network with $n_1=\lceil N/6\rceil$. By applying Lemma~\ref{lema:low} to $f_1$, we obtain, there exists an $(1,n_1)$-network $f_{1,n_1}$ such that
	\[
		\nm{f_1-f_{1,n_1}}{L^2(\Omega)}\le\dfrac{2\upsilon_{f_1,1,\Omega}}{n_1^{1/2}}\le\dfrac{2\sqrt{6}\upsilon_{f,s,\Omega}}{N^{sL}}.
	\]
Noting that an $(1,n_1)$-network can be represented by an $(L,n_1)$-network. We just need to fill the rest of the hidden layers with
	\[
		t=\begin{cases}
			\text{ReLU}(t), & t\ge 0,\\
			-\text{ReLU}(-t), & t< 0.
		\end{cases}
	\]
Secondly, we approximate $f_2$ with an $(L,n_2)$-network with $n_2=\lceil 5N/6\rceil$ and obtain the error estimate. Applying Lemma~\ref{lema:deep} we obtain, there exists an $(L,n_2)$ network $f_{2,n_2}$ such that
	\[
		\nm{f_2-f_{2,n_2}}{L^2(\Omega)}\le\dfrac{22\pi\upsilon_{f_2,s,\Omega}}{n_2^s}\le\dfrac{22\sqrt{6/5}\pi\upsilon_{f,s,\Omega}}{N^{sL}}.
	\]
These together with the triangle inequality give the estimate~\eqref{eq:thm2} and the total number of units in each layer is
	\[
		n_1+2n_2=\lceil N/6\rceil+\lceil 5N/6\rceil\le N+2.
	\]
If $f$ is a real-valued function, then we let $f_N=\mathrm{Re}(f_{1,n_1}+f_{2,n_2})$, and the upper bound~\eqref{eq:thm2} still holds true.
\end{proof}

\begin{remark}
The activation function of the last hidden layer of the $(L,N)$-network in Theorem~\ref{thm:deep} may be replaced by many other familiar activation functions such as Hyperbolic tangent, SoftPlus, ELU, Leaky ReLU, ReLU$^k$ and so on. Because all these activation functions can be reduced to sigmoidal functions by certain shifting and scaling argument; e.g., for SoftPlus, we observe that $\text{SoftPlus}(t)-\text{SoftPlus}(t-1)$ is a sigmoidal function. Unfortunately, it is not easy to change ReLU of the first $L-1$ hidden layers by other activation functions.
\end{remark}

In what follows, we shall show that Theorem~\ref{thm:deep} is sharp if the activation function of the last hidden layer is Heaviside function. This example is adopted from~\cite{BreslerNagaraj:2020}. For readers' convenience, We reserve the proof in Appendix~\ref{apd:lower}.
\begin{theorem}\label{thm:lower}
For any fixed positive integers $L,N$ and real numbers $\varepsilon,s$ with $0<\varepsilon,sL\le 1/2$, there exists $f\in\ms{B}^s(\mb{R}^d)$ satisfying $\upsilon_{f,s,\Omega}\le1+\varepsilon$ such that for any $(L,N)$-network $f_N$ whose activation function $\sigma$ in the last layer is the Heaviside function $\chi_{[0,\infty)}$, there holds
\begin{equation}\label{eq:lower}
		\nm{f-f_N}{L^2(\Omega)}\ge\dfrac{1-\varepsilon}{8N^{sL}}.
\end{equation}
\end{theorem}
\section{Conclusion}~\label{sec:conclu}
We discuss the analytical functional properties of the spectral Barron space. The sharp embedding between the spectral Barron spaces and various classical function spaces have been established. The approximation rate has been proved for the deep ReLU neural networks when the spectral Barron space with a small smoothness index is employed as the target function space. There are still some unsolved problems, such as the sup-norm error and the higher-order convergence results for larger $s$, the relations among Barron type spaces, variational space and the Radon bounded variation space as well as understanding how these spaces are related to the classical function spaces, which will be pursued in the subsequent works.
\bibliography{spectralBarron}

\appendix
\section{Some proof details}
\subsection{Proof for~\eqref{eq:phiR} and~\eqref{eq:phiRs}}\label{apd:phiR}
\begin{proof}
Note that $\wh{\phi}_R$ is a radial function. By Lemma~\ref{lema:fourinv} and $\wh{\phi}_R\in L^1(\mb{R}^d)$,
\begin{align*}
\phi_R(x)=&\int_{B_R}\Lr{1-\dfrac{\abs{\xi}^2}{R^2}}^\delta e^{2\pi ix\cdot\xi}\mr{d}\xi\\
=&\int_{-R}^Re^{2\pi i\abs{x}\xi_1}\mr{d}\xi_1\int_{\xi_2^2+\dots\xi_d^2<R^2-\xi_1^2}\Lr{1-\dfrac{\abs{\xi}^2}{R^2}}^\delta\mr{d}\xi_2\dots\mr{d}\xi_d.
\end{align*}
Performing the polar transformation and changing the variable $t=r^2/(R^2-\xi_1^2)$, we obtain
\begin{align*}
&\quad\int_{\xi_2^2+\dots\xi_d^2<R^2-\xi_1^2}\Lr{1-\dfrac{\abs{\xi}^2}{R^2}}^\delta\mr{d}\xi_2\dots\mr{d}\xi_d\\
&=\om_{d-2}\int_0^{\sqrt{R^2-\xi_1^2}}r^{d-2}\Lr{1-\dfrac{\xi_1^2+r^2}{R^2}}^\delta\mr{d}r\\
&=\dfrac{\om_{d-2}}2R^{d-1}\Lr{1-\dfrac{\xi_1^2}{R^2}}^{\delta+(d-1)/2}\int_0^1t^{(d-3)/2}(1-t)^\delta\mr{d}t\\
&=\dfrac{\om_{d-2}}2R^{d-1}\Lr{1-\dfrac{\xi_1^2}{R^2}}^{\delta+(d-1)/2}B\Lr{\dfrac{d-1}{2},\delta+1}.
\end{align*}
Substituting this equation into the previous one and changing the variable $\xi_1=R\cos\theta$, we get
\begin{align*}
		\phi_R(x)=&\dfrac{\om_{d-2}}2R^{d-1}B\Lr{\dfrac{d-1}{2},\delta+1}\int_{-R}^R\Lr{1-\dfrac{\xi_1^2}{R^2}}^{\delta+(d-1)/2}e^{2\pi i\abs{x}\xi_1}\mr{d}\xi_1\\
		=&\dfrac{\pi^{(d-1)/2}\Gamma(\delta+1)R^d}{\Gamma(\delta+(d+1)/2)}\int_0^\pi\cos(2\pi\abs{x}R\cos\theta)\sin^{2\delta+d}\theta\mr{d}\theta\\
		=&\dfrac{\Gamma(\delta+1)}{\pi^\delta\abs{x}^{\delta+d/2}}R^{-\delta+d/2}J_{\delta+d/2}(2\pi\abs{x}R),
\end{align*}
where we have used
\[
J_\nu(x)=\dfrac{(x/2)^\nu}{\pi^{1/2}\Gamma((d+1)/2)}\int_0^\pi\cos(x\cos\theta)\sin^{2\nu}\theta\mr{d}\theta,\qquad\nu>-\dfrac12
\]
in the last step. The above integral representation of the first kind of Bessel function may be found in~\cite{Luke:1962}*{\S~1.4.5, Eq. (4)}.

For $s\ge 0$, a direct calculation gives
\begin{align*}
		\upsilon_{\phi_R,s}&=\int_{B_R}\abs{\xi}^s\Lr{1-\dfrac{\abs{\xi}^2}{R^2}}^\delta\mr{d}\xi
		=\om_{d-1}\int_0^Rr^{s+d-1}\Lr{1-\dfrac{r^2}{R^2}}^\delta\mr{d}r\\
		&=\dfrac{\om_{d-1}}2R^{s+d}\int_0^1t^{(s+d)/2-1}(1-t)^\delta\mr{d}t\\
		&=\dfrac{\om_{d-1}}2B\Lr{\dfrac{s+d}2,\delta+1}R^{s+d}.
\end{align*}
Therefore $\phi_R\in\ms{B}^s(\mb{R}^d)$.    
\end{proof}

\subsection{Proof for~\eqref{eq:ex3}}\label{apd:ex3}
\begin{proof}
To prove~\eqref{eq:ex3}, we start with the following representation formula. If $\wh{f}\in L^1(\mb{R}^d)$ is a radial function with $\wh{f}(\xi)=g_0(\abs{\xi})$, then
\begin{equation}\label{eq:radial}
f(x)=\dfrac{2\pi}{\abs{x}^{d/2-1}}\int_0^\infty g_0(r)r^{d/2}J_{d/2-1}(2\pi\abs{x}r)\mr{d}r.
\end{equation}

If $d=1$, then using Lemma~\ref{lema:fourinv}, we obtain
\[
f(x)=\int_\mb{R}g_0(\abs{\xi})e^{2\pi ix\xi}\mr{d}\xi=2\int_0^\infty g_0(r)\cos(2\pi\abs{x}r)\mr{d}r,
\]
which gives~\eqref{eq:radial}, where we have used the relation~\cite{Luke:1962}*{\S~1.4.6, Eq. (7)}
\[
J_{-1/2}(x)=\sqrt{\dfrac2{\pi x}}\cos(x)
\]
in the last step.

For $d\ge 2$, combining Lemma~\ref{lema:fourinv} and~\cite{Stein:1971}*{Ch. IV, Theorem 3.3}, we obtain~\eqref{eq:radial}, which immediately implies
\begin{align*}
f_p(x)&=\dfrac{2\pi}{\abs{x}^{d/2-1}}\int_0^1 r^{d(1/2-1/p')}J_{d/2-1}(2\pi\abs{x}r)\mr{d}r\\
&=\omega_{d-1}\int_0^1r^{d/p-1}{}_0F_1(;d/2;-\pi^2\abs{x}^2r^2)\mr{d}r,
\end{align*}
where we have used the relation~\cite{Luke:1962}*{\S~1.4.1, Eq. (1)}
\[
J_\nu(x)=\dfrac{(x/2)^\nu}{\Gamma(\nu+1)}{}_0F_1(;\nu+1;-x^2/4)
\]
in the last step. Changing the variable $t=r^2$ and using the identity~\cite{Luke:1962}*{\S~1.3.2, Eq. (2)}
\[
{}_1F_2(\rho;\rho+\sigma,\beta;x)=\dfrac1{B(\rho,\sigma)}\int_0^1t^{\rho-1}(1-t)^{\sigma-1}{}_0F_1(;\beta;xt)\mr{d}t,
\]
we get~\eqref{eq:ex3}. 
\end{proof}

\subsection{Proof for Lemma~\ref{lema:Besov}}\label{apd:Besov}
\begin{proof}
The ``if''-part is standard by~\cite{Triebel:1983}*{\S~2.3.2, Proposition 2; \S~2.7.1, Theorem}. We illustrate the ``only if''-part with an example, which is taken from~\cite{Triebel:1983}*{\S~2.3.9, Proof of Theorem}. 

Let $f_0\in\ms{S}(\mb{R}^d)$ with supp$(\wh{f}_0)\subset\set{x\in\mb{R}^d}{1\le\abs{x}\le 3/2}$. Let $f_n(x)=f_0(2^{-n}x)$ for an integer $n$, then 
	\[
		\wh{f}_n(\xi)=2^{-dn}\wh{f}_0(2^{-n}\xi)\qquad \text{and}\qquad \mr{supp}(\wh{f}_n)\subset\set{x\in\mb{R}^d}{2^n\le\abs{x}\le 3\times 2^{n-1}}.
	\]
	Choose proper $\{\varphi_j\}_{j=0}^{\infty}\subset\ms{S}(\mb{R}^d)$ in the definition of Besov space such that $\varphi_0(x)=1$ when $\abs{x}\le 3/2$ and $\varphi_j=1$ on supp$(\wh{f}_j)$ for $j\ge 1$, then
	\[
		\mr{supp}(\wh{f}_n)\cap\mr{supp}(\varphi_j)=\emptyset\quad\text{if}\quad n\ge 0\quad\text{and}\quad n\neq j.
	\]
	A direct calculation gives that
	\[
		(\varphi_j\wh{f}_n)^\vee=\delta_{0n}f_n,\quad\text{if}\quad n\le 0\qquad\text{and}\qquad(\varphi_j\wh{f}_n)^\vee=\delta_{jn}\wh{f}_n,\quad\text{if}\quad n> 0.
	\]
	By definition, when $n<0$,
	\[
		\nm{f_n}{B_{p,q}^\alpha(\mb{R}^d)}=\nm{f_n}{L^p(\mb{R}^d)}=2^{-dn/p}\nm{f_0}{L^p(\mb{R}^d)}.
	\]
	Let $n\to-\infty$ with the embedding relation $B_{p_1,q_1}^{\alpha_1}(\mb{R}^d)\hookrightarrow B_{p_2,q_2}^{\alpha_2}(\mb{R}^d)$ yields $p_1\le p_2$. Similarly, when $n>0$,
	\[
		\nm{f_n}{B_{p,q}^\alpha(\mb{R}^d)}=2^{\alpha n}\nm{f_n}{L^p(\mb{R}^d)}=2^{(\alpha-d/p)n}\nm{f_0}{L^p(\mb{R}^d)}.
	\]
	Let $n\to+\infty$ with the embedding relation implies $\alpha_1-d/p_1\ge\alpha_2-d/p_2$. Finally if $\alpha_1-d/p_1=\alpha_2-d/p_2$, then $q_1\le q_2$ proved in~\cite{Triebel:1995}*{Theorem 3.2.1}. 
 \end{proof}
\subsection{Proof for~\eqref{eq:ex5}}\label{apd:ex5}
\begin{proof}
If $1\le p<\infty$ and $\alpha>0$, then $B_{p,1}^\alpha(\mb{R}^d)$ is equivalent to the space defined in~\cite{Triebel:1983}*{\S~2.5.7, Theorem}
	\[
		\Lambda_{p,q}^\alpha(\mb{R}^d){:}=\set{f\in W^{[\alpha]}_p(\mb{R}^d)}{\int_{\mb{R}^d}\dfrac{\nm{\Delta_h^2(\nabla^{[\alpha]}f)}{L^p(\mb{R}^d)}^q}{\abs{h}^{d+\{\alpha\}q}}\mr{d}h<\infty}
	\]
	equipped with the norm
	\[
		\nm{f}{\Lambda_{p,q}^\alpha(\mb{R}^d)}{:}=\nm{f}{W^{[\alpha]}_p(\mb{R}^d)}+\Lr{\int_{\mb{R}^d}\dfrac{\nm{\Delta_h^2(\nabla^{[\alpha]}f)}{L^p(\mb{R}^d)}^q}{\abs{h}^{d+\{\alpha\}q}}\mr{d}h}^{1/q}
	\]
    for $1\le q<\infty$, and
    \[
    \Lambda_{p,\infty}^\alpha(\mb{R}^d):=\set{f\in W^{[\alpha]}_p(\mb{R}^d)}{\sup_{h\in\mb{R}^d\backslash\{0\}}|h|^{-\{\alpha\}}\nm{\Delta_h^2(\nabla^{[\alpha]}f)}{L^p(\mb{R}^d)}<\infty}
    \]
    equipped with the norm
    \[
    \nm{f}{\Lambda_{p,\infty}^\alpha(\mb{R}^d)}:=\nm{f}{W_p^{[\alpha]}(\mb{R}^d)}+\sup_{h\in\mb{R}^d\backslash\{0\}}|h|^{-\{\alpha\}}\nm{\Delta_h^2(\nabla^{[\alpha]}f)}{L^p(\mb{R}^d)}.
    \]
	Here $\alpha=[\alpha]+\{\alpha\}$ with integer $[\alpha]$ and $0<\{\alpha\}\le 1$, and $\Delta_h^2f(x)=f(x+2h)-2f(x+h)+f(x)$; see~\cite{Triebel:1983}*{\S~2.2.2, Eq. (9)}. 

For any nonnegative integer $k$, a direct calculation gives
\[
\nabla^k\psi_n(x)=\sum_{j=0}^k\dfrac{c_jx^{\beta_j}}{(1+in)^{d/2+j}}e^{-\pi\abs{x}^2/(1+in)}
\]
	for some constants $\{c_j\}_{j=0}^k$, and the multi-index $\beta_j=(\beta_{j1},\dots,\beta_{jd})$ satisfies $\abs{\beta_j}\le j$ with $x^{\beta_j}=x_1^{\beta_{j1}}\dots x_d^{\beta_{jd}}$. Then
	\[
		\nm{\nabla^k\psi_n}{L^p(\mb{R}^d)}\le C\sum_{j=0}^k\dfrac1{(1+n^2)^{d/4+j/2}}\nm{\abs{x}^{\abs{\beta_j}}e^{-\pi\abs{x}^2/(1+n^2)}}{L^p(\mb{R}^d)}.
	\]
A direct calculation gives
\begin{align*}
		\int_{\mb{R}^d}\abs{x}^{\abs{\beta_j}p}e^{-\pi p\abs{x}^2/(1+n^2)}\mr{d}x=&(1+n^2)^{(d+\abs{\beta_j}p)/2}\int_{\mb{R}^d}\abs{y}^{\abs{\beta_j}p}e^{-\pi p\abs{y}^2}\mr{d}y\\
		=&\dfrac{\omega_{d-1}\Gamma\Lr{(\abs{\beta_j}p+d)/2}}{2(p\pi)^{(\abs{\beta_j}p+d)/2}}
		(1+n^2)^{(d+\abs{\beta_j}p)/2}.
	\end{align*}
Therefore, there exists $C$ depending only on $d,p,k$ such that
\begin{equation}\label{eq:ex4}
		\nm{\nabla^k\psi_n}{L^p(\mb{R}^d)}\le C\sum_{j=0}^k\dfrac{(1+n^2)^{d/(2p)+\abs{\beta_j}/2}}{(1+n^2)^{d/4+j/2}}\le C(1+n^2)^{-d(p-2)/(4p)}.
\end{equation}

If $f\in W^2_p(\mb{R}^d)$, then 
\[
		\Delta_h^2f(x)=\int_0^1\mr{d}t\int_t^{1+t}\nabla^2f(x+sh)\mr{d}sh\cdot h=\int_0^2\nabla^2f(x+sh)\mr{d}s\int_{\max(s-1,0)}^{\min(s,1)}\mr{d}th\cdot h.
\]
Therefore,
\[
\abs{\Delta_h^2f(x)}\le\abs{h}^2\int_0^2\abs{\nabla^2f(x+sh)}\mr{d}s.
\]
By the Minkowski's inequality, we obtain
	\[
		\nm{\Delta_h^2f(x)}{L^p(\mb{R}^d)}\le\abs{h}^2\int_0^2\nm{\nabla^2f(\cdot+sh)}{L^p(\mb{R}^d)}\mr{d}s=2\abs{h}^2\nm{\nabla^2f}{L^p(\mb{R}^d)}.
\]
Splitting the integral part of the $\Lambda_{p,q}^\alpha$-norm into two parts, we get
	\begin{align*}
		&\int_{\abs{h}<1}\dfrac{\nm{\Delta_h^2f}{L^p(\mb{R}^d)}^q}{\abs{h}^{d+\{\alpha\}q}}\mr{d}h+\int_{\abs{h}>1}\dfrac{\nm{\Delta_h^2f}{L^p(\mb{R}^d)}^q}{\abs{h}^{d+\{\alpha\}q}}\mr{d}h\\
		\le&2^q\nm{\nabla^2f}{L^p(\mb{R}^d)}^q\int_{\abs{h}<1}h^{(2-\{\alpha\})q-d}\mr{d}h+4^q\nm{f}{L^p(\mb{R}^d)}^q\int_{\abs{h}>1}h^{-d-\{\alpha\}q}\mr{d}h\\
		=&2^q\omega_{d-1}\Lr{\dfrac{\nm{\nabla^2f}{L^p(\mb{R}^d)}^q}{(2-\{\alpha\})q}+\dfrac{2^q\nm{f}{L^p(\mb{R}^d)}^q}{\{\alpha\}q}}\\
        \le&\dfrac{2^{2q}\omega_{d-1}}{\{\alpha\}q}\nm{f}{W^2_p(\mb{R}^d)}^q.
	\end{align*}
In the same manner we can see that when $q=\infty$,
\begin{align*}
&\sup_{h\in\mb R^d\backslash\{0\}}\abs{h}^{-\{\alpha\}}\nm{\Delta_h^2(\nabla^{[\alpha]}f)}{L^p(\mb{R}^d)}\\
\le&\sup_{0<\abs{h}<1}\abs{h}^{-\{\alpha\}}\nm{\Delta_h^2(\nabla^{[\alpha]}f)}{L^p(\mb{R}^d)}+\sup_{\abs{h}\ge 1}\abs{h}^{-\{\alpha\}}\nm{\Delta_h^2(\nabla^{[\alpha]}f)}{L^p(\mb{R}^d)}\\
\le&2\nm{\nabla^2f}{L^p(\mb R^d)}\sup_{0<\abs{h}<1}\abs{h}^{2-\{\alpha\}}+4\nm{f}{L^p(\mb R^d)}\sup_{\abs{h}\ge 1}\abs{h}^{-\alpha}\\
\le&4\nm{f}{W^2_p(\mb R^d)}.
\end{align*}
Note that $\nabla^{[\alpha]}\psi_n\in W^2_p(\mb{R}^d)$, a combination of the above inequality and~\eqref{eq:ex4} yields
\[
\nm{\psi_n}{\Lambda_{p,q}^\alpha(\mb{R}^d)}
\le C\nm{\psi_n}{W^{[\al]+2}_p(\R^d)}\le C(1+n^2)^{-d(p-2)/(4p)},
\]
where $C$ is a constant depending on $p,\alpha$ and $d$ but independent of $n$. So does $\nm{\psi_n}{B_{p,q}^\alpha(\mb{R}^d)}$.
\end{proof}

\subsection{Proof for Lemma~\ref{lema:comp}}~\label{apd:comp}
\begin{proof}
	Firstly we show that $g_{,n}$ is symmetric about $x=1/2$. By definition,
	\[
		g_{,n}(1-t)=g(n(1-t)-j)=g(nt-n+j+1)=g(nt-k)=g_{,n}(t)
	\]
for some integers $j,k$ satisfying $0\le j,k\le n-1$ and $k=n-j-1$.

For a fixed $t\in[0,1]$, there exist integers $j,k$ satisfying $0\le j\le 2n_1-1,0\le k\le n_2-1$ such that $0\le 2n_1n_2t-n_2j-k\le 1$, then $0\le n_2(2n_1t-j)-k\le 1$ and
	\[
		g_{,2n_1n_2}(t)=g(2n_1n_2t-n_2j-k)=g(n_2(2n_1t-j)-k)=g_{,n_2}(2n_1t-j).
	\]
	By definition,
	\[
		\beta_{,n}(t)=\begin{cases}
			2nt-2j,& 0\le nt-j\le 1/2,\\
			2+2j-2nt, & 1/2\le nt-j\le 1,
		\end{cases}\quad j=0,\dots,n-1.
	\]
	If $j$ is even, then $j=2l$ for some integer $l$ satisfying $0\le l\le n_1-1$ and $0\le n_1t-l\le1/2$. Therefore
	\[
		g_{,n_2}(\beta_{,n_1}(t))=g_{,n_2}(2n_1t-2l)=g_{,n_2}(2n_1t-j).
	\]
	Otherwise $j$ is odd, then $j=2l+1$ for some integer $l$ satisfying $0\le l\le n_1-1$ and $1/2\le n_1t-l\le 1$. Therefore
	\[
		g_{,n_2}(\beta_{,n_1}(t))=g_{,n_2}(2+2l-2n_1t)=g_{,n_2}(1+j-2n_1t)=g_{,n_2}(2n_1t-j).
	\]
	This gives $g_{,n_2}\circ\beta_{,n_1}=g_{,2n_1n_2}$ on $[0,1]$.
\end{proof}

\subsection{Proof for Theorem~\ref{thm:lower}}~\label{apd:lower}
\begin{lemma}\label{lema:decay}
Given $n,R>0$ and let $f(x)=\cos(2\pi nx_1)e^{-\pi\abs{x}^2/R}$, then $f\in\ms{B}^s(\mb{R}^d)$ with
	\begin{equation}\label{eq:decay}
		\upsilon_{f,s,\Omega}\le\Lr{n+\dfrac{d}{\pi\sqrt{R}}}^s\qquad\text{for}\quad 0\le s\le 1.
	\end{equation}
\end{lemma}

\begin{proof}
For any $R>0,$ the Fourier transform of the dilated Guass function $e^{-\pi x_j^2/R}$ reads as
\[
\widehat{e^{-\pi x_j^2/R}}=\sqrt{R}e^{-\pi R\xi_j^2}.
\]
A direct calculation gives 
\begin{align*}
\wh{f}(\xi)=&\dfrac12\int_{\mb{R}}e^{-\pi x_1^2/R}\lr{e^{-2\pi ix_1(\xi_1-n)}+e^{-2\pi ix_1(\xi_1+n)}}\mr{d}x_1\prod_{j=2}^d\int_\mb{R}e^{-\pi x_j^2/R-2\pi ix_j\xi_j}\mr{d}x_j\\
=&\dfrac{R^{d/2}}2\lr{e^{-\pi R(\xi_1-n)^2}+e^{-\pi R(\xi_1+n)^2}}\prod_{j=2}^de^{-\pi R\xi_j^2}\\
=&R^{d/2}e^{-\pi R(\abs{\xi}^2+n^2)}\cosh(2\pi nR\xi_1).
\end{align*}
It is clear that $f,\wh{f}\in L^1(\R^d)$ and the pointwise Fourier inversion theorem holds true, and
\[
\upsilon_{f,0,\Omega}=\int_{\R^d}\abs{\wh{f}(\xi)}\md\xi=\int_{\R^d}\wh{f}(\xi)\md\xi=f(0)=1,
\]
where we have used the positiveness of $\wh{f}$ .

Next, using the elementary identities
\[
\sqrt{R}\int_{\mb{R}}e^{-\pi R\xi_j^2}\mr{d}\xi_j=1\quad\text{and}\quad \sqrt{R}\int_{\mb{R}}\abs{\xi_j}e^{-\pi R\xi_j^2}\mr{d}\xi_j=\dfrac1{\pi\sqrt{R}},
\]
we obtain
\begin{align*}
\upsilon_{f,1,\Omega}&=R^{d/2}\int_{\mb{R}}\abs{\xi_1}e^{-\pi R(\xi_1-n)^2}\mr{d}\xi_1\prod_{j=2}^d\int_{\mb{R}}e^{-\pi R\xi_j^2}\mr{d}\xi_j\\
		&\quad+R^{d/2}\int_{\mb{R}}e^{-\pi R(\xi_1-n)^2}\mr{d}\xi_1\sum_{j=2}^d\int_{\mb{R}}\abs{\xi_j}e^{-\pi R\xi_j^2}\mr{d}\xi_j\prod_{k\neq1,j}\int_{\mb{R}}e^{-\pi R\xi_k^2}\mr{d}\xi_k\\
		&\le\sqrt{R}\int_{\mb{R}}(\abs{\xi_1-n}+n)e^{-\pi R(\xi_1-n)^2}\mr{d}\xi_1+\dfrac{d-1}{\pi\sqrt{R}}\\
&=n+\dfrac{d}{\pi\sqrt{R}}.
\end{align*}
	
Using the interpolation inequality~\eqref{eq:inter2}, we obtain~\eqref{eq:decay}.
\end{proof}
\begin{proof}[Proof for Theorem~\ref{thm:lower}]
	Define $n=2^{L+2}N^L$ and $f(x)=n^{-s}\cos(2\pi nx_1)e^{-\pi\abs{x}^2/R}$ with large enough $R$ such that $\upsilon_{f,s,\Omega}\le 1+\varepsilon$ by Lemma~\ref{lema:decay} and $e^{-\pi\abs{x}^2/R}\ge 1-\varepsilon$ when $x\in\Omega$. 
	Fix $x_2,\dots,x_d$, then any $(L,N)$-network $f_N$ can be viewed as an one-dimensional $(L,N)$-network, i.e. $f_N(\cdot,x_2,\dots,x_d):[0,1]\to\mb{C}$. Divide $[0,1]$ into $n$-internals of $[j/n,(j+1)/n]$ with $j=0,\dots,n-1$. There exists at least $n-2^{L+1}N^L=2^{L+1}N^L$ intervals such that $f_N$ does not change sign on those intervals~\cite{Telgarsky:2016}*{Lemma 3.2}. Without loss of generality, we assume $f_N(\cdot,x_2,\dots,x_d)\ge 0$ on some interval $[j/n,(j+1)/n]$, then
	\[
		\int_{j/n}^{(j+1)/n}(f(x)-f_N(x))^2\mr{d}x_1\ge\dfrac{(1-\varepsilon)^2}{n^{2s}}\int_{(4j+1)/(4n)}^{(4j+3)/(4n)}\cos^2(2\pi nx_1)\mr{d}x_1\ge\dfrac{(1-\varepsilon)^2}{4n^{2s+1}},
	\]
because $\cos(2\pi nx_1)\le 0$ when $2\pi j+\pi/2\le 2\pi nx_1\le 2\pi j+3\pi/2$.

Summing up these $n-2^{L+1}N^L$ intervals gives
	\begin{align*}
		\nm{f-f_N}{L^2(\Omega)}^2\ge&\int_{[0,1]^{d-1}}\mr{d}x_2\dots\mr{d}x_d\int_0^1(f(x)-f_N(x))^2\mr{d}x_1\\
		\ge&\dfrac{2^{L+1}N^L(1-\varepsilon)^2}{4n^{2s+1}}\ge\dfrac{(1-\varepsilon)^2}{2^{2sL+4s+3}N^{2sL}}\ge\dfrac{(1-\varepsilon)^2}{64N^{2sL}}.
	\end{align*}
Simultaneously squaring off both sides of the inequality, we obtain
\[\nm{f-f_N}{L^2(\Omega)}\ge\dfrac{1-\varepsilon}{8N^{sL}}.\]
\end{proof}
\end{document}